\documentclass[cag]{ipart_v1}

\Year{2019}
\Vol{XX}
\Issue{XX}
\firstpage{1}
\usepackage{tikz-cd}
\usepackage{amsmath}
\usepackage{tasks}
	\usepackage{graphicx}

	\newcommand{\inn}[2]{\langle #1 , #2 \rangle}
 \def\CI{{\mathcal I}}

 \newcommand{\lan}[1]{{}^{L}{#1}}
\newcommand{\nad}[1]{{}^{N}\!{#1}}

\newcommand{\qb}{ \bar{q}}

\usepackage{amscd}
\DeclareMathOperator{\PXp}{P(X^\prime, \Sigma)}
\DeclareMathOperator{\PX}{P(X, \Sigma)}

\newcommand{\tp}[1]{{}^{t}\!{#1}}
\DeclareMathOperator{\Bl }{Bl}

\usepackage[cmtip,all]{xy}
\newcommand{\pir}{\ensuremath{{\pi_{\text{red}}}}}
\newcommand{\pip}{\pi_{\text{red}}^*}

\newcommand{\calE}{{ \mathcal E}}
\newcommand{\calF}{{ \mathcal F}}

\newcommand{\calI}{{ \mathcal I}}

\newcommand{\calM}{{ \mathcal M}}

\newcommand{\calO}{{ \mathcal O}}

\newcommand{\calU}{{ \mathcal U}}

\newcommand{\MGC}{\ensuremath{\mathcal{M}_{G_{\mathbb{C}}}}}
\newcommand{\prym}{{\rm Prym}}
\newcommand{\Prym}{{\rm Prym}(S, \Sigma)}
\DeclareMathOperator{\Jac}{\Pic^0}

\def\Pic{\mathop{\rm Pic}\nolimits}

\numberwithin{equation}{section}

\newtheorem{thm}{Theorem}[]
\newtheorem{theorem}[thm]{Theorem}
\newtheorem{lemma}[thm]{Lemma}
\newtheorem{corollary}[thm]{Corollary}
\newtheorem{proposition}[thm]{Proposition}

\theoremstyle{definition}

\newtheorem{remark}[thm]{Remark}

\newtheorem{definition}[thm]{Definition}

\newtheorem{example}[thm]{Example}

\newtheorem{defn-thm}[thm]{Definition-Theorem}

\newcommand{\C}{{\mathbb C}}

\newcommand{\K}{{\mathbb K}}

\newcommand{\R}{{\mathbb R}}

\newcommand{\End}{{ End}}

\newcommand{\SL}{{ SL}}

\newcommand{\SO}{{  SO}}
\newcommand{\Sp}{{ Sp}}

\newcommand{\Id}{{Id}}

\newcommand{\SU}{{ SU}}

\newcommand{\Spec}{{ Spec}}

\newcommand{\Sym}{{ Sym}}

\newcommand{\btheorem}{\begin{theorem}}
\newcommand{\etheorem}{\end{theorem}}
\newcommand{\bproposition}{\begin{proposition}}
\newcommand{\eproposition}{\end{proposition}}
\newcommand{\bdefinition}{\begin{definition}}
\newcommand{\edefinition}{\end{definition}}
\newcommand{\bcorollary}{\begin{corollary}}
\newcommand{\ecorollary}{\end{corollary}}
\newcommand{\bproof}{\begin{proof}}
\newcommand{\eproof}{\end{proof}}
\newcommand{\bremark}{\begin{remark}}
\newcommand{\eremark}{\end{remark}}
\newcommand{\eexample}{\end{example}}
\newcommand{\bexample}{\begin{example}}

\newcommand{\elemma}{\end{lemma}}
\newcommand{\blemma}{\begin{lemma}}

\renewcommand{\bar}{\overline}

\renewcommand{\phi}{\varphi}

\newcommand{\ee}{\end{eqnarray*}}
\newcommand{\be}{\begin{eqnarray*}}

\newcommand{\beq}{\begin{equation}}
\newcommand{\eeq}{\end{equation}}

\newcommand{\bd}{\begin{enumerate}}
\newcommand{\ed}{\end{enumerate}}

\renewcommand{\hat}{\widehat}
\renewcommand{\tilde}{\widetilde}

\renewcommand{\sf}{\textsf}

\newcommand{\GC}{\ensuremath{G_{\mathbb{C}}}}
\usepackage{stmaryrd}

 \usepackage{bibentry}

\begin{document}

\title[Orthogonal Higgs bundles with singular spectral curves]{Orthogonal Higgs bundles with singular spectral curves}

\author[S.~Bradlow, L.~Branco,~L.~Schaposnik]{Steve Bradlow, Lucas Branco $\&$ Laura P. Schaposnik}

\begin{abstract} We examine Higgs bundles for non-compact real forms of $\SO(4,\C)$ and the isogenous complex group $\SL(2,\C)\times\SL(2,\C)$. This involves a study of non-regular fibers in the corresponding Hitchin fibrations and provides interesting examples of non-abelian spectral data.
\end{abstract}

\maketitle

\hfill{\it To Karen Uhlenbeck, with gratitude for her and her mathematics over the years, on the occasion of her 75th birthday.}


\setcounter{section}{0}

\section{Introduction}


The celebrated Hitchin fibration introduced in \cite{N2} is a map from $\mathcal{M}(G)$, the moduli space of $G$-Higgs bundles on a closed surface  $\Sigma$ to a vector space defined by holomorphic differentials on the surface (see {\it Section \ref{fibration}} for details).  Through the eigenvalues of the Higgs field, each point in the base defines a curve called a {\it spectral curve} contained in the total space of the canonical bundle $K$.  The generic fibers are abelian varieties (in many cases Jacobian or Prym varieties  defined by the spectral curves) obtained via the eigenspaces of the Higgs field.  The fibration has many remarkable features, not the least of which is that exhibits the mirror symmetry proposed by Strominger-Yau-Zaslow  \cite{syz}.  
A fundamental feature of this fibration, namely the compactness of its fibers, is a direct consequence of Karen Uhlenbeck's seminal result on  connections with $L^p$ bounds on curvature \cite{uhlenbeck1982}.  We thus hope that what follows, in which the Hitchin fibration plays a prominent role, may be viewed  as an appropriate tribute to Karen and her mathematics.

\smallbreak

In \cite{iso} we used a fiber product construction on spectral curves to describe the relation between spectral data for generic fibers in the Hitchin fibrations on a pair $\mathcal{M}(G)$ and $\mathcal{M}(G')$, where $G$ and $G'$ are isogenous Lie groups.  Our definition of genericity ensured that all spectral curves were smooth.  We considered two cases coming from special isomorphisms between low-dimensional Lie algebras, namely for $\SO(4,\C)$ and $\SO(6,\C)$. Our fiber product construction for smooth spectral curves has since been generalized  (see \cite{emilio3}) to study $G$-Higgs bundles for products of groups $G=G_1\times G_2$.

\smallbreak

In this paper we move beyond the generic case and explore some singular fibers in the fibrations for the pair of groups $(G=\SL(2,\C)\times\SL(2,\C)$ and $G=\SO(4,\C))$.  Our choice of which fibers to examine is dictated by the non-compact real forms $G_r$ for those groups. 
  For any pair $G_r\subset G$, the moduli space $\mathcal{M}(G_r)$ defines a subvariety of $\mathcal{M}(G)$. If $G_r$ is a split real form then $\mathcal{M}(G_r)$ intersects every fiber of the Hitchin fibration for $\mathcal{M}(G)$ but this is not so in the case of other real forms \cite{thesis}. Indeed in the cases we consider, the intersection with generic fibers is empty. 
\smallbreak

The singular fibers we examine are those that contain the Higgs bundles for the non-split real forms of our isogenous pair (see {\it Section \ref{isogen}} for details). 
Our overall goal is to  examine how the fiber product construction extends to  singular situations. For this  we consider the Hitchin fibrations for the complex groups and describe the full fibers over the points in the base corresponding to our singular spectral curves, and identify their intersections with the Higgs bundle moduli spaces for the real forms.  In particular, we see in {\it Section \ref{case1}} and {\it Section \ref{case2}} that the latter are abelian varieties for which we get explicit geometric descriptions.  
\smallbreak

An alternative approach to our questions is provided by so-called cameral data for Higgs bundles. Pioneered by Donagi and Gaitsgory for complex $G$ (see \cite{donagi1993decomposition} and  \cite{doga}) the methods have recently been extended to real forms (see \cite{GPPNR, GPPN}).  While in some ways this is a natural language for comparing Higgs bundles related to different representations of a given group, it is more abstract than the concrete, more geometric, descriptions of spectral data, at least for generic fibers of Hitchin fibrations. It remains an interesting challenge to harness the power of the abstract machinery to address the specific situation described above and to compare the results with our approach. 
 
\smallbreak

\subsubsection*{Summary of main methods and results} Our basic construction takes a pair of $\SL(2,\C)$-spectral curves, say $S_1$ and $S_2$, and builds an $\SO(4,\C)$-spectral curve from the fiber product $S_1\times_{\Sigma} S_2$.  By the work of Simpson \cite{simpson92}, the points in the fiber of a Hitchin fibration can be identified with rank one torsion-free sheaves on the spectral curve. In the regular locus, where  $S_1$ and $S_2$ are smooth and have disjoint ramification divisors,  the sheaves are line bundles and a natural box-product construction relates the line bundles for the $\SL(2,\C)\times\SL(2,\C)$-Higgs bundles to the $\SO(4,\C)$- spectral line bundles \cite{iso}.  The corresponding constructions for the non-smooth spectral curves require more intricate algebraic geometry of torsion free sheaves on singular curves. 
  
 \smallbreak

The case of $\SL(2,\C)$ and $\SO(1,3)$ is studied in {\it Section \ref{case1}}, where $\SL(2,\C)$ is regarded as a real form of $\SL(2,\C)\times\SL(2,\C)$, isogenous to the real form $SO_0(1,3)\subset SO(4,\C)$.  The image of the isogeny map at the level of Higgs bundles is described in  Proposition \ref{propo1} and Proposition \ref{propo2}.
\smallbreak

  \noindent {\bf Theorem A.}   {\it Restricted to the $SL(2,\C)$-Higgs bundles from Definition \ref{defsl}, the isogeny \eqref{I2} gives a map
 \begin{eqnarray}\label{AA}
\CI_2\left(V\oplus V, {\tiny \left(\begin{matrix}\varphi&0\\0&-\varphi\end{matrix}\right)}\right) = \left(\mathcal{O}\oplus Sym^2V, \left(\begin{matrix}0&\beta\\-\beta^T&0\end{matrix}\right)\right), \label{42}
\end{eqnarray} 
where $\beta=\begin{bmatrix}-c&2a&b\end{bmatrix}$ when $\varphi={\tiny \begin{bmatrix}a&b\\c&-a\end{bmatrix}}$, and the orthogonal structure on $Sym^2V$ comes from the orthogonal structure on $V$.
\smallbreak

Moreover, if ${\tiny \left(\mathcal{O}\oplus W, \Phi\right)}$ represents a point in  the component$$ \mathcal{M}^{+}(\SO_0(1,3))\subset\mathcal{M}^{+}(\SO(4,\C))$$  then $W=Sym^2V$ for $V$  a rank two bundle with trivial determinant, and the orthogonal structure on $W$  as above.
  }
  
\smallbreak
Whilst the  generic intersection of  $SO(3,1)$-Higgs bundles  with the fibres of the $SO(4,\C)$ Hitchin fibration was described in \cite[Section 7]{cayley} where it is shown that it admits   the structure of an
abelian group, one further refinement leads to a more precise description for the identity component  of $SO_0(3,1)$-Higgs bundles which appear in the image of the isogeny map:\smallbreak
  
\smallbreak

 \noindent {\bf Proposition \ref{teo1}}  {\it  The locus of $SO_0(1,3)$-Higgs bundles $(E,\Phi)$ inside the fibre of the Hitchin fibration over $$\det (\eta-\Phi) = \eta^2(\eta^2-4q),$$ where $q\in H^0(\Sigma , K^2)$  is  a quadratic differential with only simple zeros,  is isomorphic to the dual of the Prym variety, which   for $S:=\{\eta^2 - q = 0\}$ is given by \[\prym (S, \Sigma)/ \Pic^0(\Sigma)[2].\]}
 
   \smallbreak
 \smallbreak
When 
  considering $\SU(2)\times\SL(2,\R)$, the spectral curves $S_1$ and $S_2$  are the non-reduced multiplicity two curve $2\Sigma$ and a generically smooth double cover $S$. In this case, as shown in {\it Section \ref{case2}}, the fiber product construction yields a singular curve $2\Sigma\times_{\Sigma}S$ which is a split ribbon on the curve $S$.  Moreover, we show in Proposition \ref{propo40} that one has a blow-up map  $
+ : 2\Sigma \times_\Sigma S \to 2S\ ,
$ through the restriction of the plus map.
In this case we find that our  fiber product construction leads to the following (see Proposition \ref{propso4}, Proposition \ref{propo40}, Proposition \ref{propo41} and Proposition \ref{propo44}):
 \smallbreak 
 
 \smallbreak
 
  \noindent {\bf Theorem B.}   {\it The isogeny $\mathcal{I}_2$ induces a map on   $\SU(2)\times\SL(2,\R)$-Higgs bundles   whose image is 
    \begin{eqnarray}
  \mathcal{I}_2((U,0),(N\oplus N^*,\Phi))=(UN\oplus UN^{*}, Q, I\otimes\Phi)\in \mathcal{M}(SO^*(4)),\nonumber
\end{eqnarray}   
  where   the orthogonal structure $Q$ comes from the symplectic structures on $U$ and $N\oplus N^*$.}  
   {\it In the  fibers of the $SL(2,\C)\times SL(2,\C)$-Hitchin fibration over points  giving the  spectral curve $S=\{0=\eta^2-q\}$ of an $\SU(2)\times\SL(2,\R)$-Higgs bundles with spectral data $(2\Sigma, \calE)$ and $(S,L)$ for $L\in \Prym[2]$,   the spectral data corresponding to the $\SO^*(4)$-Higgs bundle in the image of $\CI_2$ is   $(\tilde S,\tilde{\mathcal{L}})$, where
\begin{itemize}
\item the curve $\tilde S= +(2\Sigma\times_{\Sigma}S)$, and 
\item the line bundle $\tilde{\mathcal{L}}=+_*(\calE \boxtimes L)$,
for $\boxtimes$ as defined as in Proposition \ref{propo41}.
\end{itemize}}

The spectral data for $SO^{*}(4))$-Higgs bundles $(E,\Phi)$ was considered in \cite{nonabelian} as an example of {\it nonabelianization} of spectral data for real Higgs bundles: the spectral data was described there as given by a curve defined through $\sqrt{\det(\Phi-\eta Id)}$ and a rank 2 vector bundle on it satisfying certain conditions, and whose direct image is $E$. The relation between this perspective and that of Theorem B. is explored in Proposition \ref{propo42}.

       Finally, we should mention that the isogeny between $\SU(2)\times\SL(2,\R)$ and $\SO^*(4)$, and the induced maps between the corresponding Higgs bundles, turn out to be a useful device for exploring expected consequences of mirror symmetry, which we shall do across the paper.  Recall that Kapustin-Witten \cite{Kap} used  Higgs bundles to obtain a physical derivation of the {\it geometric Langlands correspondence} through mirror symmetry.  In this setting,  Higgs bundle moduli spaces for a complex Lie group $G_\C$ and its Langlands dual group $\lan{G_\C}$ are a mirror pair in the sense of the SYZ conjecture \cite{syz}, where the two Hitchin fibrations give the dual fibrations:  the non-singular fibers of $\calM (G_\C)$ and $\calM (\lan{G_\C})$ are dual abelian varieties.

       Given a real form $G$ of a complex reductive group $G_\C$, Higgs bundles in $\calM (G_\C)$ admitting a structure of $G$-Higgs bundles form a complex Lagrangian subvariety, which is an example of (the support) of what physicists call a $(B,A,A)$-brane. According to mirror symmetry, $(B,A,A)$-branes on $\calM (G_\C)$ should correspond to branes of type $(B,B,B)$ on $\calM (\lan{G_\C})$, which are hyperholomorphic branes (see e.g. \cite{Kap}).  Moreover,  the support of this dual brane $(B,B,B)$-brane  has been conjectured by Baraglia and Schaposnik \cite{slices} to correspond to Higgs bundles for the associated Nadler group  \cite{slices}. We shall comment on this when considering each moduli space of real Higgs bundles.

\section{Higgs bundles and the Hitchin fibration}\label{fibration}\label{intro}
 Throughout the paper we shall consider a fixed Riemann surface
 $\Sigma$  over $\C$ of genus $g\ge 2$. The canonical bundle of $\Sigma$ shall be denoted by $\pi: K\to \Sigma$,  its total space by $|K|$, and $\eta \in H^0(|K|, \pi^*K)$ its tautological section.

\subsection{General features}\label{general}
 
Throughout the paper we consider Higgs bundles for complex groups $G_\C$ as well as for real forms of these groups.   We will need only the special cases where $G_\C$ is $\SL(n,\C)$ or $\SO(n,\C)$ but for the sake of completeness we include here some of the main general definitions following \cite{N1,N2}.    

\begin{definition}\label{complex}\label{defcomplex} For any complex reductive group $G_\C$, a $G_\C$-Higgs bundle on $\Sigma$ is given by a pair $(P_{G_{\C}},\Phi)$ for $P_{G_{\C}}$ a principal $G_\C$-bundle on $\Sigma$, and the Higgs field $\Phi$ a holomorphic section of ${\rm Ad}P_{G_{\C}}\otimes K$, for ${\rm Ad}P_{G_{\C}}=P_{G_{\C}}\times_{Ad}\mathfrak{g}_\C$  the adjoint bundle associated to $P_{G_{\C}}$.  
 \end{definition}

In order to construct $G_r$-Higgs bundles for a real form $G_r$ of $G_\C$,  fix $H$ a maximal compact subgroup of $G_r$,  and the Cartan decomposition  
$\mathfrak{g}=\mathfrak{h}\oplus \mathfrak{m},$
for $\mathfrak{h}$ the Lie algebra of $H$ and $\mathfrak{m}$ its orthogonal complement. 
Through the induced isotropy representation $\text{Ad}|_{H^{\mathbb{C}}}: H^{\mathbb{C}}\rightarrow GL(\mathfrak{m}^{\mathbb{C}})$, we can define:

   \begin{definition}\label{realform}
 Given $G_r$ a real form of a complex reductive Lie group $G_\C$, and with $H_{\C}$ and $\mathfrak{m}^{\mathbb{C}}$ as above, a  $G_r$-Higgs bundle on $\Sigma$ is a pair $(P_{H_{\C}},\Phi)$ where  $P_{H_{\C}}$ is a holomorphic principal $H^{\mathbb{C}}$-bundle on $\Sigma$, and 
  $\Phi$ is a holomorphic section of $P_{H_{\C}}\times_{Ad}\mathfrak{m}^{\mathbb{C}}\otimes K$.
\label{defrealhiggs}
\end{definition}
 
\begin{remark}\label{vect-data}
Following \cite{N2} if $G_\C\subseteq\SL(n,\C)$, (or is a real form of $G$) we may replace   $(P_{G_{\C}}, \Phi)$ with the Higgs pair  $(E,\Phi)$  where $E$ is a rank $n$ holomorphic vector bundle, the Higgs field $\Phi:E\rightarrow E\otimes K$ is a holomorphic endomorphism, and the pair $(E,\Phi)$ satisfies extra conditions determined by the group.  For example, if $G_\C=\SL(n,\C)$,  the extra conditions are $\det(E)\simeq\mathcal{O}$ and $\mathrm{Tr}(\Phi)=0$.    
\end{remark}

For $G$ either complex reductive or a real form thereof,  there is a notion of stability inspired by GIT which permits the construction of moduli spaces $\mathcal{M}(G)$ of Higgs bundles, whose points represent equivalence classes of semi-stable objects.  Equivalently, $\mathcal{M}(G)$ may be constructed as gauge-equivalence classes of solutions to gauge-theoretic equations known as the Hitchin equations, and the equivalence of the two constructions is the result of a suitable version of the Hitchin-Kobayashi correspondence.

A  powerful way of studying the moduli spaces   $\MGC$ of $ G_\C$-Higgs bundles  is through the Hitchin fibration  \cite{N2}.   To define it, consider  a homogeneous basis  of invariant polynomials $\{p_{1}, \ldots, p_k\}$ on the Lie algebra of $ {G_\C}$, of degrees $\{d_{1}, \ldots, d_k\}$. Then, the Hitchin fibration is given by
\begin{eqnarray} \label{hfibration}h~:~ \mathcal{M}_{ {G_\C}}\longrightarrow\mathcal{A}_{ {G_\C}}:=\bigoplus_{i=1}^{k}H^{0}(\Sigma,K^{d_{i}}), 
\end{eqnarray} 
where  $h : (E,\Phi)\mapsto (p_{1}(\Phi), \ldots, p_{k}(\Phi))\nonumber$ is referred to as the {\it Hitchin~map}: it is a proper map for any choice of basis. 

The fibers of this fibration can be described using {\it spectral data} for Higgs bundles.  From this perspective, a $ {G_\C}$-Higgs bundle  $(E,\Phi)$ defines a {\it spectral curve},   an algebraic curve $S=\{{\rm det}( \Phi - \eta Id)=0\} \subset |K|$ which is a ramified covering of $\Sigma$ in which the preimage of a point $x\in\Sigma$ encodes the eigenvalues of the Higgs field $\Phi(x)$. We say that $(E,\Phi)$ lies in the {\em regular locus} of $\mathcal{M}({G_\C})$ if the curve $S$ is non-singular. The fibers in the regular locus are called the {\em regular fibers} of the Hitchin fibration.

The generic fibres of the Hitchin fibration are abelian varieties (the Jacobian $\Jac(S)$ or abelian subvarieties), and the Higgs pairs $(E,\Phi)$ are recovered from the spectral data as follows: given a line bundle $L\in \Jac(S)$, one recovers a Higgs bundle  by considering $E=\pi_*L$ and $\Phi$  as the direct image of the tautological section of $\pi^*K$. But adding extra conditions to the spectral curve or the line bundle, one recovers Higgs bundles for subgroups of $GL(n,\C)$ -- see \cite{N2} for classical complex groups, \cite{BNR} for a more general setting (including integral curves), and \cite{thesis,mas1,mas2,mas3,mas4,mas5} for the case of real forms of these groups. 
More generally, away from the regular locus, Simpson showed in \cite{simpson} that the line bundle $L$ must be replaced by a rank one torsion free sheaf supported on $S$ and that the fibers of the Hitchin fibration can be identified with moduli spaces of semistable sheaves of this type.  This will be important in Sections \ref{case1} and \ref{case2} where we are forced to confront singular spectral curves.

\begin{remark}\label{fixedpoint}The inclusion of a real form in its complexification as the fixed point locus of an  anti-holomorphic involution induces an anti-holomorphic involution on the moduli space $\mathcal{M}(G_\C)$, through which $\mathcal{M}(G_r)\subset \mathcal{M}(G_\C)$ lies as the fixed point locus  -- from this perspective that one can see real Higgs bundles within the Hitchin fibration (see \cite{N5,thesis}, and references therein).  Unless needed,  we will not distinguish   $\mathcal{M}(G_r)$ from its image in $\mathcal{M}(G_\C)$.
\end{remark}
 
 \subsection{$\SL(n,\C)$-Higgs bundle}\label{SLHiggs} \label{firstSLnC}  
  
From Definition \ref{defcomplex} and Remark \ref{vect-data}, an $\SL(n,\C)$-Higgs bundle $(E,\Phi)$ consists of a rank $n$ holomorphic bundle  $E$  with trivial determinant, and a Higgs field $\Phi$ for which $\mathrm{Tr}(\Phi)=0$.  
The non-compact real forms of $\SL(n,\C)$ are the split form $\SL(n,\R)$ and the special unitary groups with indefinite signature $\SU(p,q)$ with $p+q=n$.   
In the case of $\SL(2,\C)$, the real forms $\SL(2,\R)$ and $\SU(1,1)$ are isomorphic.  Since no other cases of $SU(p,q)$ will be needed we describe just the case $\SL(2,\R)$, first described in \cite{N5}.    
\begin{remark}From Definition \ref{realform} and Remark \ref{vect-data}, an  $\SL(2,\R)$-Higgs bundle can be described by a pair $(E,\Phi)$ where $E=L\oplus L^*$ for a holomorphic line bundle $L$ on $\Sigma$, and the Higgs field $\Phi={\tiny \left(\begin{array}{cc}0&\beta\\ \gamma&0\end{array}\right)}$ where the maps $\beta\in H^0(\Sigma, L^{-2} \otimes K)$ and $\gamma\in H^0(\Sigma, L^2\otimes K).$ \label{sl2r}\end{remark}

From Remark \ref{sl2r} an $\SL(2,\R)$-Higgs bundle is a  triple $(L,\beta,\gamma)$. When  $\deg(L)=0$,  the Higgs bundle is always semistable; otherwise the semistability condition reduces to the condition that $\beta\ne 0$ (if $\deg(L)<0$) or  $\gamma\ne 0$ (if $\deg(L)>0$).  Then, it follows that $|\deg(L)|\le g-1$ if $(L,\beta,\gamma)$ is semistable.

The Higgs field in an $\SL(n,\C)$-Higgs bundle defines a spectral curve 
\begin{equation}
S:=\{0=\det(\Phi-\eta Id)\}=\{0=\eta^{n}+a_{2}\eta^{n-2}+a_3\eta^{n-3}+\ldots +a_{n-1}\eta+a_n\}\nonumber
\end{equation}
 where $a_i\in H^0(\Sigma,K^i)$, and $\eta$ is the tautological section of $\pi^*K$.  The fibers in the regular locus of the Hitchin fibration are isomorphic to Prym varieties ${\rm Prym}(S,\Sigma)$, where the Prym conditions required for ${\rm det}E\cong \mathcal{O}$  \cite{N1}.  In the case of $\SL(n,\R)$-Higgs bundles, it is shown in \cite[Theorem 4.12]{thesis} that they correspond to torsion two points in the fibres of the Hitchin fibration.

 \subsection{$\SO(n,\C)$-Higgs bundle}\label{SOHiggs} \label{firstSOnC}  

From   Definition \ref{defcomplex} and Remark \ref{vect-data},  an $SO(n,\mathbb{C})$-Higgs bundle may be defined by the data $((E,Q,\delta),\Phi)$ where  $(E,Q, \delta)$ denotes a holomorphic vector bundle of rank $n$ with a non-degenerate symmetric bilinear form $Q$ and a compatible trivialization $\delta$ of the determinant bundle $\Lambda^{n}E$, and \linebreak
 $\Phi\in H^{0}(\Sigma,{\rm End}(E)\otimes K)$ is a Higgs field which is skew-symmetric with respect to $Q$.  If no confusion will result we sometimes suppress the bilinear form $Q$ and the trivialization $\delta$, and denote the $SO(n,\mathbb{C})$-Higgs bundle by the pair $(E,\Phi)$.
 
The moduli space $\mathcal{M}({SO(n,\mathbb{C})})$ has two connected components, which we shall write as  $\calM^\pm_{SO(2m,\C)}$, labeled  by the second Stiefel-Whitney class $w_{2}(E) \in H^{2}(\Sigma,\mathbb{Z}_{2}) \cong  \mathbb{Z}_{2}$ \cite{N5}, i.e.
\[\calM_{SO(2m,\C)} = \calM^+_{SO(2m,\C)}\cup \calM^-_{SO(2m,\C)} \subseteq \calM_{SL(2m,\C)}. \] 
As in \cite{lucas}, we shall let  $\mathcal{M}^+_{SO(n,\mathbb{C})}$ be the  component where $w_{2}(E)=0$ or, equivalently, in which   $E$ has a lift to a spin bundle.  

The non-compact real forms of $\SO(n,\C)$ include the groups $\SO(p,q)$ with $p+q=n$ and, if $n$ is even, the group $\SO^*(n)$. In this paper we consider only $\SO(4,\C)$.  In this case,   $SO(4,\C)$ has three  real forms, namely $\SO(1,3)$, the split real form $\SO(2,2)$, and $\SO^*(4)$.  Setting aside the split real form (which we discuss in \cite{iso}) we are thus left with  $\SO(1,3)$ and $\SO^*(4)$, which we shall discuss in later sections.

 The spectral data for $SO(n,\mathbb{C})$-Higgs bundles is described in \cite{N2} and \cite{N3} (see also \cite{cayley}).  In the case of interest in this paper, namely $SO(4,\mathbb{C})$-Higgs bundles,  the  characteristic polynomials are of the form
 \begin{eqnarray}\label{sopoly}
{\rm det}(\Phi - \eta Id)=  \eta^{4}+b_{1}\eta^{2}+b^2_{2} 
\end{eqnarray}
 where $b_1,b_2\in H^0(\Sigma, K^{2})$.  Thus the base of the $\SO(4,\C)$-Hitchin fibration is $\mathcal{A}_{\SO(4,\C)}=H^0(\Sigma,K^2)\oplus H^0(\Sigma,K^2)$.  Whilst the curves defined by \eqref{sopoly} are necessarily singular,  for generic choices of $a_1$ and $b_1$ there are smooth curves defined by their normalization, and 
the smoothness of these normalized curves defines the regular locus in the base of the Hitchin fibration, and the regular fibres   are given by Prym varieties of  those smooth curves.  
 
 \begin{remark}\label{so4convention}
 We will use the convention that the spectral curve in the fiber $h^{-1}_{\SO(4,\C)}(b_1,b_2)$ is defined by the vanishing of the polynomial in \eqref{sopoly}. \end{remark}

\subsection{Hitchin fibers on ribbons}

In Sections \ref{case1}-\ref{case2} we are forced to consider non-integral spectral curves, and in particular the  so-called ribbons.  We describe here some features of such curves and of the rank 1 torsion-free sheaves which they support (c.f. \cite{lucas}). 

\begin{definition} Let $S$ be a non-singular irreducible projective curve over $\C$. A \textbf{ribbon} $X$ on $S$ is a curve  such that  $X_{\text{red}} \cong S,$
the ideal sheaf $\calI$ of $S$ in $X$ satisfies
$\calI^2 = 0,$
and  $\calI$ is an invertible sheaf on $S$.
\end{definition}

In this paper we encounter two specific examples of ribbons, both of which arise as spectral curves.  The first example is the spectral curve in the nilpotent cone of the $\SL(2,\C)$-Hitchin fibration, i.e. the ribbon on $\Sigma$ defined by the characteristic equation $\eta^2=0$.

In the second example that we encounter $X$ is the  non-reduced inside $|K|$ defined by the zeros of $(\eta^2 - \pi^*q)^2 \in H^0(|K|, \pi^*K^4)$.  As shown in \cite{nonabelian},  ribbons of this type arise as spectral curves for certain $G_r$-Higgs bundles.  The full fibers over these points in the Hitchin fibrations for the complex groups have been completely described in \cite{lucas}. In what follows we shall recall from \cite{lucas}  the description of the singular fibre $h_{SL(4,\C)}^{-1}(X)$ (resp. $h_{SO(4,\C)}^{-1}(X)$) of the Hitchin fibration for $SL(4,\C)$ (resp. $SO(4,\C)$).

 We will assume that $q\in H^0(\Sigma , K^2)$ has only simple zeros. In this case, $X$ is a ribbon over the smooth curve $S=X_{red}$ and the ideal sheaf corresponding to the natural inclusion $i: S\hookrightarrow X $ is isomorphic to $K_S^{*}\cong \pi^*K^{-2}$. Denoting the restrictions of $\pi:K \to \Sigma$ to $X$ by the same letter,  one has a finite flat morphism of degree $4$, and by restricting to the reduced scheme one has a ramified double covering $\pir : S\to \Sigma$.

\subsubsection{The \texorpdfstring{$SL(4,\C)$ case}{Lg}}\label{sl4sect}  In                                                                                                                                                                                                                                                                                                                                                                                                                                                                                                                                                                                                                                                                                                                                                                                                                                                                                                                                                                                                                                                                                                                                                                                                                                                                                                                                                                                                                                                                                                                                                                                                                          this section we assume that $S$ is a smooth spectral curve and $X$ is a ribbon on $S$.   The fibre $h_{SL(4,\C)}^{-1}(X)$ is isomorphic to Simpson's moduli space parametrizing semi-stable rank $1$ torsion-free sheaves $\calE$ on $X$ satisfying $\det (\pi_{*} \calE)\cong \calO_\Sigma$. Rank $1$ torsion-free sheaves on $X$ are either rank $2$ vector bundles on the reduced curve or generalized line bundles on $X$ (i.e. torsion-free sheaves which are invertible in a Zariski open subset) and the fibre has several irreducible components. 

\begin{thm}[See \cite{lucas} Theorem 0.2]\label{fibresu} The fibre $h_{SL(4,\C)}^{-1}(X)$   is a disjoint union of $N(SL(4,\C)) = \{ E\in \calU_S(2,4m(g-1)) \ | \ \det (\pir_{,*}E)\cong \calO_\Sigma )\}$ and spaces $A_{d}$ for $d=1,\ldots,4(g-1)$, where  $A_d \to Z_d$ is an affine bundle  on
\[Z_d =  \{ (L,D)\in \Pic^{d_2}(S)\times S^{(\bar{d})} \ | \ L^2(D)\otimes\pip K^{-3}\in \Prym \}\] 
 in which the fibre at $(L,D) \in Z_d$ is isomorphic to $H^1(S,K_S^{*}(D))$, for the degrees $\bar{d} = -2d + 2(g_{_S} -1)$ and $d_2 = (k-\bar{d})/2$.
\end{thm}

The components of $h_{SL(4,\C)}^{-1}(X)$ can be understood in terms of the corresponding Higgs bundles: indeed, consider Higgs bundles $(E,\Phi)$ such that   $\det(\Phi-\eta Id)=(\eta^2-q)^2$.  In this case one has $(\Phi^2-\Id_V\otimes q)^2=0$ but  when $\Phi^2-\Id_V\otimes q=0$, the corresponding rank $1$ torsion-free sheaf is supported on $S$ and thus must be locally-free of rank $2$. Hence, the rank two sheaves in $N$  correspond to Higgs bundles $(V,\Phi)\in h_{SL(4,\C)}^{-1}(X)$ such that  $\Phi^2-\Id_V\otimes q: V \to V\otimes K^2$ vanishes identically. 

In order to understand the spaces $\mathcal{A}_d$, note that if $\mathcal{L}\in h_{SL(4,\C)}^{-1}(X)\setminus N$ then the corresponding Higgs bundle can be expressed as an extension of Higgs bundles 
$
0 \to \mathsf{W}_1 \to \mathsf{V} \to \mathsf{W}_2 \to 0,
$
where  $\mathsf{W}_i = (W_i,\Phi_i) = \pir_*(L_i, \eta)$  for  $i=1,2$ are rank $2$ Higgs bundles corresponding to the line bundles $L_i$ on $S$ of degree 
$d_i = (-1)^id+2(g-1)$, with $L_1L_2\pip K\in \Prym $ and $0< d = \deg (W_2) = - \deg (W_1)\le g_S -1 = 4(g-1)$.  These extensions  are in bijection with the space of infinitesimal deformations of the Higgs bundle $\mathsf{W}_2^*\otimes \mathsf{W}_1$, i.e., each such $\mathsf{V}$ corresponds to an element $\delta = \delta (\mathsf{V})$ of the first hypercohomology group $\K^1 = \K^1 (\Sigma , \mathsf{W}_2^*\otimes \mathsf{W}_1)$.
Since $\K^1$ fits into the short exact sequence 
$
0\to H^1(S,L_2^*L_1) \to \K^1 \xrightarrow{\varsigma} H^0(S, L_2^*L_1K_S) \to 0
$
one has that \[
A_d = \left \{ \ \delta \in \K^1  \
\bigg |
\gathered
\begin{array}{cl}
i. & \mathsf{W}_i = \pir_*(L_i, \eta),\text{ where }L_i\in \Pic^{d_i}(S),\\
ii. & L_1L_2\pip K\in \Prym , \\
iii. & \varsigma (\delta)\neq 0. \ 
\end{array}
\endgathered  \ \  \right
\}.\]
In particular, setting $\bar{d} = \deg (L_2^*L_1K_S) = -2d + 8(g-1),$ 
we have the natural map
$p_d : A_d  \to Z_d$.
Thinking of the affine bundle $A_d$ as the space of generalized line bundles $\calE$ (with $\det (\pi_* \calE)\cong \calO_\Sigma$) which fail to be invertible in some effective divisor $D_\calE$ on $S$ of degree $\bar{d}$, we can interpret the composition 
$A_d \to Z_d \to S^{(\bar{d})},$
in terms of (non-proper) Prym varieties.

\begin{proposition}[See \cite{lucas} Proposition 0.5]The fibre of the map $A_d \to S^{(\bar{d})}$ at an effective divisor $D$ is a torsor for the Prym variety $\PXp$, where $X^\prime$ is the blow-up of $X$ at $D$ (where the Cartier divisor $D$ is considered as a closed subscheme of $X$). In particular, the Prym variety $\PX$ is isomorphic to $A_{4(g-1)}$ and it has $2^{2g}$ connected components.
\end{proposition}


\subsubsection{The \texorpdfstring{$SO(4,\C)$}{Lg} case}\label{so4sect}
  For a generic point $X$ in the $SO(4,\C)$ Hitchin base, one  may see $h^{-1}_{SO(4,\C)}(X)$ as a closed subscheme of $h^{-1}_{SL(4,\C)}(X)$, which in turn is isomorphic to Simpson's moduli space parametrizing semi-stable rank $1$ torsion-free  
 sheaves $\calE$ on $X$ satisfying $\det (\pi_{*} \calE)\cong \calO_\Sigma$ \cite{simpson88}.
Then, one may recover the components corresponding to $SO(4,\C)$ Higgs bundles by imposing some extra conditions on the data defining the $SL(4,\C)$-fibre.  

The closed subscheme of $h^{-1}_{SO(4,\C)}(X)\cap N(SL(4,\C))$ corresponds to 
\[
N(SO(4,\C)) = \left \{ \ E \in \calU_S (2,4(g-1)) \
\left|
\begin{array}{l}
   \exists~ \psi : \sigma^*E \to E^* \otimes \pi^*K^{2m-1} \\
  \text{isomorphism~s.t.}~\tp{(\sigma^*\psi)} = - \psi.\
\end{array}
\right.\right
\},\]     
and we define $N_0(SO^*(4)):= \{ E\in N(SO(4,\C)) \ | \ \det (E)\cong \pip K \}
$.
  In the other components, the orthogonal structure and smoothness of the reduced spectral curve $S$ impose constraints on extensions of Higgs bundles.

\begin{thm}[See \cite{lucas} Theorem 0.3]\label{fibreso}   The fibre $h_{SO(4,\C)}^{-1}(X)$ is a disjoint union of 
$N(SO(4,\C))$ and   $A_d(SO(4,\C))$ for $d=1,\ldots,2(g-1)$, where $N(\SO(4,\C))$ is as above, and  $A_d(SO(4,\C))$ are  affine bundles on 
\begin{equation*}
Z_d(SO(4,\C)) =  \{ (M, D^\prime) \in   \Pic^{d_2}(S)\times \Sigma^{(d^\prime)} \ | \ M \sigma^*M \cong \pip  (K^2(-D^\prime))  \},
\end{equation*}
in which the fiber at $(M,D^\prime)$ is isomorphic to $H^1(\Sigma, K^{*}(D^\prime))$. 
\end{thm} 

Moreover, $N(SO(4,\C))$ has $2$ connected components $N^{\pm}(SO(4,\C))$ corresponding to the parity of  $\sigma$-isotropic line subbundles, and thus to the parity of the second Stiefel-Whitney class of the bundle underlying the corresponding Higgs bundle. In particular, $N_0(SO^*(4))$ is contained in $N^+(SO(4,\C))$, where Higgs bundles have $w_2=0$.  As shown in \cite{lucas}, each stratum has dimension $\dim SO(4,\C)(g-1)$, and 
\begin{align*}
h_{SO(4,\C)}^{-1}(X) \cap \calM^\pm(SO(4,\C))  = N^\pm(SO(4,\C)) \cup \bigcup_{d=1}^{g-1} A_{d_\pm}(SO(4,\C)), 
\end{align*} 
for $d_+=2d$ and $d_-=2d-1$. 
The irreducible components of $h_{SO(4,\C)}^{-1}(X)$ are   the Zariski closures of $A_d(SO(4,\C))$, for $1 \leq d \leq 2(g-1)$, and $N(SO(4,\C))$.

\begin{proposition}\label{SOPrym}[See \cite{lucas} Proposition 0.7] Let  $D$ be the effective divisor $R_{\pir} + \pip D^{\prime}\in S^{(\bar{d})}$ of $S$, where $R_{\pir}$ is the ramification divisor. At an effective divisor $D^\prime \in \Sigma^{(d^\prime)}$, the fibre of  $A_d(SO(4,\C)) \to Z_d(SO(4,\C)) \to \Sigma^{(d^\prime)} $
 is a torsor for the Prym schemes $\prym (Y_D, \bar{Y}_D)$ for $\bar Y_D = \Bl_D(X)$    the geometric quotient of by the  involution  induced by $\sigma : \eta\mapsto -\eta$. 
 \end{proposition}

\section{The Isogenies and their induced maps}\label{isogen}

If $\omega$ is the usual complex symplectic form on $\C^2$ then the tensor product defines an isogeny  
\begin{eqnarray}\calI_2: SL(2,\C)\times SL(2,\C) &\to SO(4,\C),\label{I2}\label{iso13}\\
(g_1,g_2)& \mapsto g_1\otimes g_2\nonumber
\end{eqnarray}
where $SO(4,\C) = SO(\C^2\otimes\C^2, \omega\otimes \omega)$, i.e. the group of isomorphisms of $C^4\simeq\C^2\otimes\C^2$ which preserve the symmetric bilinear form $Q=\omega\otimes \omega$.

We are interested in the restriction of  $\calI_2$  to the real forms of the complex groups, and the induced morphisms on the moduli spaces of Higgs bundles.  The isogenous pairs of real forms are shown in the following table:

  \begin{center}
\begin{tabular}{clcr}
(compact)& $\SU(2)\times\SU(2)$ &$\sim$&$\SO(4)$; \\
(split)& $ \SL(2,\R)\times\SL(2,\R)$ &$\sim$& $\SO(2,2)$;\\
& $\SU(2)\times\SL(2,\R)$ &$\sim$ &$\SO^*(4)$; \\
&  $\SL(2,\C)$&$\sim$& $\SO(1,3)$. \\
\end{tabular}
\end{center}
 
In the case of a compact real form,   the Higgs fields in the corresponding Higgs bundles are identically zero, and thus a $G$-Higgs bundle is simply a holomorphic principal $G_{\C}$-bundle.  The case of the split real forms was treated in \cite{iso}. In this paper we consider the remaining two pairs of real forms.  

We shall describe the  induced maps on Higgs bundles  in Section \ref{iso11},  and on spectral data for smooth points of the Hitchin fibration in Section \ref{iso2}, to then study the singular locus in Sections \ref{case1} and \ref{case2}.

\subsection{The isogenies and Higgs bundles}\label{iso11}

Recall that a homomorphism $f:G_1\rightarrow G_2$ between complex groups induces a map on Higgs bundles  $(E,\Phi)$, since it induces a map on principal bundles and on the adjoint representations. This correspondence descends to a map between the Higgs bundle moduli spaces.  The case of interest to us here is the map $\mathcal{I}_2$ given in \eqref{I2}, for which the induced map is
\begin{equation}\
((V_1,\varphi_1),(V_2,\varphi_2))  \mapsto (V_1\otimes V_2, \omega_1 \otimes \omega_2, \varphi_1\otimes \Id + \Id \otimes \varphi_2).\label{mapa2} 
\end{equation}
 Here
\begin{itemize}
\item $((V_1,\varphi_1),(V_2,\varphi_2))$ denotes a pair of $\SL(2,\C)$-Higgs bundles - i.e. an $\SL(2,\C)\times\SL(2,\C)$-Higgs bundle - and 
\item $(V_1\otimes V_2, \omega_1 \otimes \omega_2, \varphi_1\otimes \Id + \Id \otimes \varphi_2)$ is an $\SO(4,\C)$-Higgs bundle where $V_1\otimes V_2$ is the rank four vector bundle, $\omega_1 \otimes \omega_2$ is the orthogonal structure determined by the symplectic structures $\omega_i$ on the trivial determinant rank two bundles $V_i$, and the third entry is the Higgs field.
\end{itemize}
 As observed in \cite{iso}, this map induces a finite morphism (generically $2^{2g}$:$1$) 
\begin{equation}
\mathcal{I}_2:\calM (SL(2,\C)) \times  \calM (SL(2,\C)) \to \calM^+(SO(4,\C)) \label{isomod}
\end{equation}
from the moduli space of $Spin (4,\C)$-Higgs bundles onto the connected component of the moduli space of $SO(4,\C)$-Higgs bundles containing the trivial element. 

.

\subsection{The isogenies   and  the Hitchin fibration}\label{second}\label{iso2}
 
From the description of Higgs bundles in Section \ref{SLHiggs}, one has that 
the spectral curve for an $\SL(2,\C)$-Higgs bundle is determined by   $q\in H^0(K^2)$ via 
$S= \{ \eta^2 - q=0 \} \subset |K|,\label{curve SL2}
$
where $\pi:K\rightarrow\Sigma $.  Thus for an $\SL(2,\C)\times\SL(2,\C)$-Higgs bundle we have a pair of curves $(S_{1},S_{2})$ defined by a pair $(q_1,q_2)\in H^0(K^2)\oplus H^0(K^2)$. 
The  the spectral curve $S_4$ of an  $\SO(4,\C)$-Higgs bundle is defined   by a pair $(a_2,b_2)\in  H^0(K^2)\oplus H^0(K^2)$ via the normalization of $
S_{12}=  \{ \eta^4 +a_2\eta^2+b_2^2 = 0 \} \subset |K|$ as in \eqref{sopoly}.
In \cite{iso} we show that  the map induced by  $\mathcal{I}_2$ takes $(S_{1},S_{2})\longmapsto S_{12}$ where
$
a_2=-2(q_1+q_2)$ and $b_2^2=(q_1-q_2)^2.
$
On the regular fibers this map, as well as the relation between the pair $(S_{1},S_{2})$ and the smooth normalization of $S_{12}$, can be described by taking the fiber product of the two $\SL(2,\C)$-spectral curves and the induced spectral data from $L_i\in \prym(S_i,\Sigma)$.  The resulting curve lies in the total space of $K\oplus K$. The curve $S_4\subset |K|$ and the corresponding spectral data are obtained by applying the fiber-wise sum
\begin{equation}\label{plusmap}
|K\oplus K| \xrightarrow{+} |K|.
\end{equation}

\begin{thm}[\cite{iso}]  \label{hola1}For $S_1$ and $S_2$   smooth and with disjoint ramification divisors on $\Sigma$,  the product   $S_1 \times_\Sigma S_2$ is smooth,  the curve $S_{12}$ in  \eqref{sopoly} is a nodal curve and the restriction of the $+$ operation defines the normalization of $S_{12}$.
 Moreover, the spectral data for the  image   $\mathcal{I}_2((E_1,\Phi_1), (E_2,\Phi_2))$ is  $(\hat{S}_4,\mathcal{L})$ where 
  $\hat S_4:=S_1\times_{\Sigma}S_2$ is a smooth ramified fourfold cover, and
  $\mathcal{L}:= p_1^*(L_1)\otimes p_2^*(L_2)$ where $p_i:S_1\times_{\Sigma}S_2\rightarrow S_i$ are the projection maps.
\end{thm}
\begin{remark}[\cite{iso}]  The map described in Theorem \ref{hola1} takes torsion two points to torsion two points, and thus gives a correspondence between Higgs bundles for split real forms. 
\end{remark}

 In the next two sections we give an analogue to Theorem \ref{hola1} for non regular fibres of the Hitchin fibration, corresponding to  the following curves:
\begin{enumerate}
\item The characteristic curve of an $SO(1,3)$-Higgs bundles appearing when $S_1=2\Sigma$ and $S_2$ generic  (see Section \ref{case1}).
\item The characteristic curve of an $SO^*(4)$-Higgs bundles appearing when $S_1=S_2$.  (see Section \ref{case2}).
\end{enumerate}

We will be interested in the restriction of the plus map \eqref{plusmap}  to the fibre product $S_1 \times_\Sigma S_2$, which
  is a curve sitting inside $|K| \times_\Sigma|K| \cong |K\oplus K|$. 
In preparation for these   general cases we shall  further examine  \eqref{plusmap}.

\begin{proposition} \label{prop1}  For $i=1,2$ let $S_i= \{0=\eta^2-q_i\} \subset|K|$ be spectral curves defined by any quadratic differentials $q_i\in H^0(K^2)$,  and let $S_{12}\subset |K|$ be the curve defined by 
$\eta^4 -2(q_1+q_2)\eta^2+(b_2)q_1-q_2)^2 = 0$.
 The sum operation \eqref{plusmap} gives a well-defined morphism 
$
+ : S_1 \times_\Sigma S_2 \to S_{12}.
$ 
Moreover, $S_1 \ncong S_2$ if and only if the scheme theoretic image of the map above is $S_{12}$.
\end{proposition}

\begin{proof} 
Consider  an open affine subset $U\subset \Sigma$ and  a nowhere vanishing section $u \in H^0(U, K^{*})$, so that $| K|_U | = \Spec (R[u])$, where $R = \calO_\Sigma (U)$. Given the local function  
$ \bar{q}_i = \langle q_i,u^2\rangle \in R$ 
induced by the natural pairing on $U$ between $K^2$ and $K^{-2}$,  we obtain the affine open set 
$ \pi^{-1}(U) = \Spec \left(\frac{R[u]}{u^2 - \bar{q}_i} \right),$ 
on $S_i$ and its restriction to $\pi$ is  given by the inclusion $R \hookrightarrow R[u]/(u^2-\qb_i)$. Similarly, we have the open affine subset { 
\[ \pi^{-1}(U) = \Spec \left( \frac{R[u]}{u^4-2(\bar{q}_1+\bar{q}_2)u^2+(\bar{q}_1 - \bar{q}_2)^2}  \right) \subseteq S_{12}\]      }
mapping onto $\Spec (R)$ via $\pi$.

The sum operation restricts locally to the following ring morphism
\[
+^\sharp : \frac{R[u]}{u^4-2(\bar{q}_1+\bar{q}_2)u^2+(\bar{q}_1 - \bar{q}_2)^2} \to \frac{R[u]}{u^2-\bar{q}_1}\otimes_R \frac{R[u]}{u^2-\bar{q}_2}\]
defined as $u \mapsto u_1 + u_2,$
where $u_1 = u\otimes 1$ and $u_2=1\otimes u$. A direct calculation gives 
$
u^2 \mapsto 2u_1u_2 + \bar{q}_1+\bar{q}_2, $ and $u^4 \mapsto 4(\bar{q}_1+\bar{q}_2)u_1u_2 + \bar{q}_1^2+\bar{q}_2^2 + 6\bar{q}_1\bar{q}_2,$
so that $u^4-2(\bar{q}_1+\bar{q}_2)u^2+(\bar{q}_1 - \bar{q}_2)^2$ is mapped to zero and the map is indeed well-defined. For the second part it is enough to determine the kernel of $+^\sharp$. Suppose $+^\sharp (p(u)) = 0$. We can represent $p(u)$ as a polynomial of degree no greater than $3$ of the form
$p(u) = a_0 + a_1u + a_2u^2 + a_3 u^3,$ 
where $a_j \in R$, $j=0,1,2, 3$. Since 
$u^3 \mapsto u_1 (\qb_1 + 3\qb_2) + u_2 (\qb_2 + 3\qb_1),$
the condition $p(u)\in \ker (+^\sharp)$ gives  
$(a_0 + (\qb_1 + \qb_2)a_2) + (a_1 + (\qb_1 + 3\qb_2)a_3)u_1$ $+ (a_1 + (\qb_2 + 3\qb_1)a_3)u_2 +$ $ (2a_2)u_1u_2 = 0.$
Hence, $a_0 = a_2 = 0$, and $a_1 = - (\qb_1 + 3\qb_2)a_3 $ and 
$a_3(\qb_1 - \qb_2) = 0.$ 
Since $R$ is a domain ($\Sigma$ is integral), we can either have $a_3 = 0$ or $\qb_1 = \qb_2$. Since $q_i$ can only vanish at a finite number of points, it follows that $S_1 \ncong S_2$ (i.e. $q_1\neq q_2$) is equivalent to $a_3=0$ (which in turn is equivalent to $a_1=0$ and also to the map $+^\sharp$ being always injective). The scheme theoretic image of $+$ is $S_{12}$ if and only if $S_1\ncong S_2$. 
\end{proof}

\begin{remark}
The curve $S_{12}$ is  connected. If $div (q_1) \cap div (q_2) \neq \emptyset $, then $S_1 \times_\Sigma S_2$ is connected. 
\end{remark}
 

\newpage
\section{The real forms $SL(2,\C)$ and  $SO_0(1,3)$}\label{case1}

In this section we consider Higgs bundles for $\SL(2,\C)$ as a real form of $\SL(2,\C)\times\SL(2,\C)$ and for the real form $\SO(1,3)$ of $\SO(4,\C)$.
Our main goal is understand the map between them induced by the isogeny between the real forms and, in particular, to give a geometric description of the induced map on spectral data.

\subsection{The Higgs bundles}

   Consider the anti-holomorphic involution on $\SL(2,\C)\times\SL(2,\C)$ defined by 
\begin{eqnarray}\sigma (A,B)=((B^*)^{-1},(A^*)^{-1}).\label{invo1}\end{eqnarray}
 The fixed points of $\sigma$ are given by $\SL(2,\C)\subset \SL(2,\C)\times\SL(2,\C)$  embedded via   $A\mapsto (A, (A^*)^{-1})$. Under this embedding, the maximal compact subgroup  $\SU(2)$ maps to $H=\{(A, A)\ |\ A\in\SU(2)\}$, a maximal compact subgroup of $\SL(2,\C)\times\SL(2,\C)$.  Then, from  Definition \ref{realform} and Remark \ref{vect-data}:

\begin{definition}\label{defsl}
A Higgs bundle defined by the real form 
$\SL(2,\C)$ fixed by \eqref{invo1} of $\SL(2,\C)\times\SL(2,\C)$
is an $\SL(2,\C)\times\SL(2,\C)$-Higgs bundle of the form $((V,\varphi),(V,-\varphi))$ where $(V,\varphi)$ defines an $\SL(2,\C)$-Higgs bundle. Viewed as an $\SL(4,\C)$-Higgs bundle it has the form $(E,\Phi)$ with  $E=V\oplus V$ and $\Phi= {\tiny \begin{pmatrix}
\varphi & 0 \\
0 & -\varphi  
\end{pmatrix}}$.
\end{definition}

\begin{definition}\label{defso13} An $\SO(1,3)$-Higgs bundle is defined by a triple $(L,W,\beta)$ where $L$ and $W$ are orthogonal vector bundle of ranks one and three respectively, with $\det(W)\simeq L$, and $\beta : W \to L\otimes K$. Viewed as an $\SO(4,\C)$-Higgs bundle it has the form $(E,\Phi)$ where $E=L \oplus W$ and  $\Phi = {\tiny \begin{pmatrix}
0 & \beta \\
-\beta^T & 0 
\end{pmatrix}},$ with the orthogonal structure on $E$ defined by the direct sum of the orthogonal structures on the summands, and the transpose $\beta^T$ taken with respect to the bilinear forms defining the orthogonal structures on $L$ and $W$.
\end{definition}

\begin{remark}\label{so13asso4}
Restricting to the connected component of the identity, i.e. $\SO_0(1,3)$, the resulting $SO_0(1,3)$-Higgs bundles must have $L=\mathcal{O}$. Such Higgs bundles are thus defined either by triples $(\mathcal{O},W,\beta)$ or by pairs  $(E,\Phi)$ as above but with $I=\mathcal{O}$.  Assuming that the orthogonal structure on $\calO_\Sigma$ is given by $2\in H^0(\Sigma, \calO_\Sigma)$,   the orthogonal structure on $\calO\oplus W$ is then given by ${\tiny Q = \begin{pmatrix}
2 & 0 \\
0 & q_W  
\end{pmatrix}}$, where $q_w$ defines the orthogonal structure on $W$.
\end{remark}

\subsubsection{The induced map on the Higgs bundles}

In what follows we shall consider the maps induced by the isogeny  \eqref{I2} both on Higgs bundles and on their moduli spaces. 
 \begin{proposition}\label{propo1}
 Restricted to the $SL(2,\C)$-Higgs bundles from Definition \ref{defsl}, the isogeny \eqref{I2} gives a map
 \begin{eqnarray}\label{I2onSL2}
\CI_2\left(V\oplus V, {\tiny \left(\begin{matrix}\varphi&0\\0&-\varphi\end{matrix}\right)}\right) = \left(\mathcal{O}\oplus Sym^2V, \left(\begin{matrix}0&\beta\\-\beta^T&0\end{matrix}\right)\right), \label{42}
\end{eqnarray} 
where $\beta=\begin{bmatrix}-c&2a&b\end{bmatrix}$ when $\varphi={\tiny \begin{bmatrix}a&b\\c&-a\end{bmatrix}}$, and the orthogonal structure on $Sym^2V$ comes from the orthogonal structure on $V$.
\end{proposition}

\begin{proof}
 The isogeny \eqref{I2} maps an  $SL(2,\C)$-Higgs bundle $\left(V\oplus V, {\tiny \left(\begin{matrix}\varphi&0\\0&-\varphi\end{matrix}\right)}\right)$  to the $SO(4,\C)$ Higgs bundle $(V\otimes V, \varphi\otimes I-I\otimes \varphi)$.  After applying the isomorphism $V\otimes V\simeq \Lambda^2V\oplus Sym^2V=\mathcal{O}\oplus Sym^2V
$, we get \eqref{42}
 where the orthogonal structures on $\Lambda^2V=\mathcal{O}$ and $Sym^2V$ come from the structures $Q=\omega\otimes\omega$ on $V\otimes V$, 
and $\beta^T=q_3^{-1}\beta^t q_1$ where $q_1$ and $q_3$ are the quadratic forms defining the orthogonal structures on $\mathcal{O}$ and $Sym^2V$ respectively.

To understand the induced map on the Higgs field, fix local frames $\{e,f\}$ for $V$ such that $\omega={\tiny \begin{bmatrix}0&1\\ -1&0\end{bmatrix}}$.  Consider the local frames $\mathcal{F}_1=\{e\otimes e, e\otimes f, f\otimes e, f\otimes f\}$ and $\mathcal{F}_2=
\{e\otimes f-f\otimes e, e\otimes e, e\otimes f+ f\otimes e, f\otimes f\}$, giving the isomorphism $V\otimes V =\mathcal{O}\oplus Sym^2V$. Then, with respect to the local frames $\mathcal{F}_1$ and   $\mathcal{F}_2$,  the orthogonal structure on $V\otimes V$ is defined by the quadratic forms $Q_1$ and $Q_2$, respectively, given by
\begin{equation} \label{Q1Q3}{\tiny 
Q_1=\begin{bmatrix}0&0&0&1\\
0&0&-1&0\\
0&-1&0&0\\
1&0&0&0
\end{bmatrix}~{\rm~ and~}~
Q_2=\begin{bmatrix}2&0&0&0\\
0&0&0&1\\
0&0&-2&0\\
0&1&0&0
\end{bmatrix}=\begin{bmatrix}q_1&0\\0&q_3\end{bmatrix}}.
\end{equation}
Let $\varphi={\tiny \begin{bmatrix}a&b\\c&-a\end{bmatrix}}$ with respect to  the local frame $\{e,f\}$, and denote the image Higgs field by  $\Phi:=\varphi\otimes I-I\otimes \varphi$.  Then, with respect to $\mathcal{F}_2$, the Higgs field is locally given by $\Phi={\tiny \begin{bmatrix}0&\beta\\-\beta^T&0\end{bmatrix}}$ for $\beta=\begin{bmatrix}-c&2a&b\end{bmatrix}$. 
 \end{proof}

\label{SL2modmaps}

One should note that the moduli space $\mathcal{M}(\SL(2,\C)$ is connected but the moduli space $\mathcal{M}(\SO_0(1,3))$ has components labelled by the second Stiefel-Whitney class of the rank three orthogonal bundle.  We denote these by $\mathcal{M}^{\pm}(\SO_0(1,3))$, where $+$ corresponds to the component where $w_2$ vanishes.

\begin{proposition} \label{propo2} If ${\tiny \left(\mathcal{O}\oplus W,  \begin{pmatrix}
0 & \beta \\
-\beta^T & 0 
\end{pmatrix}\right)}$ represents a point in  the component $ \mathcal{M}^{+}(\SO_0(1,3))\subset\mathcal{M}^{+}(\SO(4,\C))$  then $W=Sym^2V$ for $V$  a rank two bundle with trivial determinant, and the orthogonal structure on $W$  as in Proposition \ref{propo1}.
\end{proposition}
The above proposition follows from the previous description of $\SO_0(1,3)$-Higgs bundles, and implies the following:
\begin{corollary} 
There is a natural commutative diagram 
\begin{equation}\label{modmaps}
\xymatrix{
\mathcal{M}(\SL(2,\C)\times\SL(2,\C))\ar[r]^{\mathcal{I}_2}& \mathcal{M}^{+}(\SO(4,\C))\\
\mathcal{M}(\SL(2,\C))\ar[u]\ar[r]^{\mathcal{I}_2}& \mathcal{M}^{+}(\SO_0(1,3))\ar[u]},
\end{equation}where the vertical maps are given by inclusions, and the horizontal maps are induced by \eqref{iso13}. Moreover,  the map to  $\mathcal{M}^{+}(\SO_0(1,3)$ is surjective.
\end{corollary}

\subsection{Spectral data} 
The  Hitchin base for the moduli spaces  $\mathcal{M}(\SL(2,\C)\times\SL(2,\C))$ and for $\mathcal{M}(\SO(4,\C))$ is  $H^0(K^2)\oplus H^0(K^2)$, and the corresponding Hitchin maps are, respectively, 
\begin{equation}\label{h22}
(V_1,\varphi_1),(V_2,\varphi_2))\mapsto(\det(\varphi_1),\det(\varphi_2))~{\rm and}~
 (E,\Phi)\mapsto(a_2,\mathrm{Pf}(\Phi)^2),
\end{equation}
where $a_2$ is the coefficient of $\eta^2$ in the characteristic polynomial of the Higgs field $\det(\eta I-\Phi)$.  The maps in 
\eqref{modmaps} preserve the fibers of the Hitchin fibrations. Direct calculation shows:

\begin{proposition}  The restrictions of the Hitchin fibrations  \eqref{h22}   to  the real moduli spaces $\mathcal{M}(\SL(2,\C))$ and $\mathcal{M}^{+}(\SO_0(1,3))$ are, respectively,\linebreak
$
((V,\varphi),(V,-\varphi))\mapsto(\det(\varphi),\det(\varphi)),
$ 
 and
\begin{equation}
\left(\mathcal{O}\oplus W,  \Phi=\begin{pmatrix}
0 & \beta \\
-\beta^T & 0 
\end{pmatrix}\right)\mapsto (-\beta\beta^T,0)=(Tr(\Phi^2),0).
\end{equation}
\end{proposition}

The following diagram summarises these maps and shows the maps induced by the isogeny $\mathcal{I}_2$ on the base of the Hitchin fibrations. The map $h_{\SO_0(1,3)}$ is given by $h_{\SO_0(1,3)}[\mathcal{O},W,\eta]=-\beta\beta^T$.
\\
   \begin{eqnarray}\label{diagram}
\scalebox{.76}{\xymatrix{
&\mathcal{M}(\SL(2,\C)\times\SL(2,\C))\ar[rr]^{\mathcal{I}_2}\ar'[d][dd]^(.4){h_{\SL(2,\C)\times\SL(2,\C)}}&&\mathcal{M}^+(\SO(4,\C))\ar[dd]^{h_{\SO(4,\C)}}\\
\mathcal{M}(\SL(2,\C))\ar@{^{(}->}[ru]\ar[rr]^(.6){\mathcal{I}_2}\ar[dd]_{h_{\SL(2,\C)}}&&\mathcal{M}^{+}(\SO_0(1,3))\ar@{^{(}->}[ru]\ar[dd]^(.3){h_{\SO_0(1,3)}}&\\
&H^0(K^2)\oplus H^0(K^2) \ar'[r]_(.8){(q_1,q_2)\mapsto(2(q_1+q_2),q_1-q_2)}[rr]&&H^0(K^2)\oplus H^0(K^2)&\\
H^0(K^2) \ar[ru]^{q\mapsto (q,q)}\ar[rr] _{q\mapsto 4q}&&H^0(K^2)\ar[ru]_{q\mapsto (q,0)}&
}}.\end{eqnarray}
Through the diagram \eqref{diagram} one can see  the image under the Hitchin fibration of the moduli spaces 
$\mathcal{M}(\SL(2,\C))$ and $\mathcal{M}(\SO_0(1,3))$.  In particular, they intersect only the fibers over points of the form $(q,q)$ in the base for $\mathcal{M}(\SL(2,\C)\times\SL(2,\C))$,  and of the form $(q,0)$ in the base for $\mathcal{M}(\SO(4,\C)$. 
In this section we shall study those fibres, their intersections with the real moduli spaces,  and the relation between these spaces induced by the isogeny.

We restrict attention to the generic case, by which we mean the case in which the spectral curve defined by $q\in H^0(K^2)$ is smooth.  Under this assumption, the spectral data is easy to describe for the fiber over $(q,q)$ and its intersection with $\mathcal{M}(\SL(2,\C))$, as we shall see below.  This is not true on the $\SO(4,\C)$ side where the fiber product $S\times_{\Sigma}S$ is not smooth, even if $S$ is smooth. We thus cannot automatically apply the results of \cite{iso}.  Furthermore the point $(q,0)$ lies in the discriminant locus for the $\SO(4,\C)$-Hitchin fibration. Indeed, the spectral curve for the Higgs bundles over $(q,0)$ are defined by the equation $\eta^2(\eta^2-q)=0$, and are thus never smooth.

The fibers $h_{\SO(4,\C)}^{-1}(q,0)$ can in principle be understood using Simpson's work (e.g. see \cite{simpson92}).  Moreover the spectral data for $\SO_0(1,3)$-Higgs bundles has been described in \cite{cayley}, and the results in \cite{lucas} for ribbon curves apply to the factor $\eta^2=0$ in the spectral curves defined by $(q,0)$. 
Let $S=\{
\eta^2-\det(\varphi)=0.\}\subset |K|$  be the spectral curve for the $\SL(2,\C)$-Higgs bundle $(V,\varphi)$.    
If $S$ is smooth, i.e. for generic sections $q=\det(\varphi)\in H^0(K^2)$, the spectral data for $(V,\varphi)$ consists of the pair $(S,L)$ where $L\in \prym(S,\Sigma)$ and thus the statements below follow. 

\begin{proposition}\label{SS}
The real $\SL(2,\C)$-Higgs bundles of Definition \ref{defsl}  seen as $\SL(2,\C)\times\SL(2,\C)$-Higgs bundles have spectral curves $(S,S)$.  If $S$ is smooth,  the rest of the spectral data is given by a pair $(L, L^{*})$ where $L$ is a line bundle in $Prym(S,\Sigma)$.
\end{proposition}

\begin{proof}The spectral curve $S$ is a two-fold ramified cover of $\Sigma$ with involution $\sigma:S\rightarrow S$ given by $\sigma(\eta)=-\eta$.  The line bundle $L$ satisfies $\sigma^*(L)\simeq L^*$ and also $\pi_*(L\pi^*(K^{1/2})\simeq V$.  Thus the spectral data $(S,\sigma^*(L))$ defines the $\SL(2,\C)$-Higgs bundle $(\pi_*(\sigma^*(L)\pi^*(K^{1/2})),-\varphi)$. Since $\sigma^*(L)\simeq L^*$, it nowfollows from relative duality and the natural isomorphism $V^*\simeq V$ that $\pi_*(\sigma^*(L)\pi^*(K^{1/2}))\simeq V^*\simeq V$. Thus the pair $(S,\sigma^*(L))=(S,L^*)$ gives spectral data for $(V,-\varphi)$.
\end{proof}

Through the isogeny $\mathcal{I}_2$ one can see a few other properties of the real Higgs bundles and their images. 

\begin{proposition}For a generic differential $q\in H^0(K^2)$, the fibre of the $\SL(2,\C)\times\SL(2,\C)$-Hitchin fibration over $(q,q)$ is isomorphic to the space
$ \Prym \times \Prym.$
Moreover, the intersection of the fibre with the real slice of $\SL(2,\C)$-Higgs bundles is given by
\[\{(L,L^{*})\ |\ L\in \prym(S,\Sigma)\}\\
\simeq \prym(S,\Sigma)\\
\simeq h^{-1}_{\SL(2,\C)}(q).
\]
\end{proposition}

 An $\SO_0(1,3)$-Higgs bundle $(E,\Phi)$ as in Remark \ref{so13asso4}  has characteristic polynomial $
\det(\Phi-\eta Id)=\eta^2(\eta^2-Tr(\Phi^2))$. If $(E,\Phi)=\mathcal{I}_2(V,\varphi)$ then $Tr(\Phi^2)=4\det(\varphi)$, and thus we shall consider  $S'=\{\eta^2-4\det(\varphi)=0\}$. Using the convention described in Remark \ref{so4convention}, the spectral curve for the fiber $h^{-1}_{\SO(4,\C)}(-4\det(\varphi),0)$ is thus  the reducible, non-reduced curve $2\Sigma\cup S'$. 

\begin{proposition}\label{teo1} The locus of $\SO_0(1,3)$-Higgs bundles inside the fibre $h_{\SO(4,\C)}^{-1}(q,0)$ is isomorphic to $\prym (S, \Sigma)/ \Pic^0(\Sigma)[2]$ where $S=\{\eta^2 - q = 0\}$, i.e. to the dual of the Prym variety $\prym (S, \Sigma)$.
\end{proposition}

\begin{proof} The surjectivity of the induced map $\mathcal{I}_2$ shows that there is a surjective map from $\prym (S, \Sigma)$ to the locus of all $\SO_0(1,3)$-Higgs bundles inside the fibre $h_{\SO(4,\C)}^{-1}(q,0)$. The fibers of this map can be identified with $\Pic^0(\Sigma)[2]$ because $VL\otimes VL\simeq V\otimes V$ if $L\in \Pic^0(\Sigma)[2]$, and $V$ is of rank two.
\end{proof}

\begin{proposition}\label{propo3}\label{phinotphi2} Let $(L,W,\beta)$ be any $\SO_0(1,3)$-Higgs bundle representing a point in the image  $\mathcal{I}_2(\mathcal{M}(\SL(2,\C)))$, and let $(E,\Phi)$ be the corresponding $\SO(4,\C)$-Higgs bundle. Then the Higgs field $\Phi$ satisfies
\begin{equation}\label{reducedCH}
    \Phi(\Phi^2-Tr(\Phi^2)Id)=0.
\end{equation}
\end{proposition}

\begin{proof}Higgs bundles in the image of the map $\mathcal{I}_2$ are  of the form $(E,\Phi)$ with $E=V\otimes V$ and $\Phi = \varphi\otimes \Id - \Id\otimes \varphi$ for some $SL(2,\C)$-Higgs bundle $(V,\varphi_i)$. Straightforward computation shows that
\begin{eqnarray}\Phi^3=\varphi^3\otimes \Id+3(\varphi\otimes\varphi^2-\varphi^2\otimes\varphi)-\Id\otimes\varphi^3.\end{eqnarray}
But $\varphi^2=q\Id$ by the Cayley-Hamilton theorem, where $q=\det(\varphi)$. Hence $\Phi^3=4q\Phi$. It remains to verify that $Tr(\Phi^2)=4q$, but this can be done using the descriptions of $\Phi$ and $\varphi$ with respect to local frames for the bundles as in the proof of Proposition \ref{propo1}.
\end{proof}

\begin{remark}
In fact Proposition \ref{phinotphi2} can be generalized to any $\SO(4,\C)$-Higgs bundle in the image of the map $\mathcal{I}_2$ and in the fiber over a point of the form $(4q,0)$ in the base of the Hitchin fibration. Such Higgs bundles are of the form given in \eqref{mapa2}, with $\varphi_i^2-q Id=0$ for $i=1,2$. 
\end{remark}
\begin{remark}
The Cayley-Hamilton theorem implies that $\Phi$ satisfies the condition $\Phi^2(\Phi^2-Tr(\Phi^2)Id)=0$. The fact that it already satisfies \eqref{reducedCH} will be significant later - see Remark \ref{Snot2S}.
\end{remark}



\subsection{Interpretation of the maps induced by the isogeny $\mathcal{I}_2$}

Under the assumption that all spectral curves are smooth the fiber product construction in \cite{iso} provides a mechanism for understanding the map induced on spectral data by the isogeny $\mathcal{I}_2:\SL(2,\C)\times\SL(2,\C)\rightarrow\SO(4,\C)$. These smoothness assumptions are not valid for the $\SL(2,\C)$- and $\SO_0(1,3)$-Higgs bundles. We now explore to what extent the fiber product mechanism accounts for $\SO_0(1,3)$-spectral data described in the previous sections. 

\subsubsection{The spectral curves}

Let $S$ be the spectral curve defined by the $\SL(2,\C)$-Higgs bundle $(V,\varphi)$ and let $S'$ be as in  the previous sections. Recall that $S$ is a double cover of $\Sigma$ with involution $\sigma$ which switches the sheets. In view of Proposition \ref{SS} we need to consider the fiber product $S\times_{\Sigma}S$. This differs in key ways from the generic case of two distinct smooth curves, where $S_1\times_{\Sigma}S_2$ is smooth, including:

\begin{enumerate}
    \item $S\times_{\Sigma}S\subset K\oplus K$ has two irreducible components. One, denoted by $\Delta_+$, comes from the diagonal embedding of $S$ given by $y\rightarrow (y,y)$, and the other, $\Delta_-$, comes from the embedding $y\rightarrow (y,\sigma(y))$.
    \item In addition to the commuting involutions $\sigma_1=\sigma\times 1, \sigma_2=1\times \sigma$,  and $\sigma_{12}=\sigma_1\sigma_2$, $S\times_{\Sigma}S$ has the involution which switches the factors, i.e. $\sigma_3(x,y)=(y,x)$. 
 \end{enumerate}

Since $S\times_{\Sigma}S$ lies in the total space of  $K\oplus K$, it does not directly define a spectral curve. The relation to a spectral curve in the total space of $K$ comes from the following Proposition.

\begin{proposition} \label{propo29}
Scheme-theoretically,  $+(S\times_{\Sigma}S)=\Sigma\cup S'$.  
 \end{proposition}

\begin{proof}   In order to see that $+(S\times_{\Sigma}S)=\Sigma\cup S'$, note that locally $S\times_\Sigma S$ is $\Spec (A) $, where $ A = R[u,v]/(u^2 - \bar{q}, v^2 - \bar{q})$.   The scheme-theoretic image of the map $+$ can be understood from 
$+^\sharp: \frac{R[u]}{u^2(u^2 - 4\bar{q})} \to \frac{R[u,v]}{u^2 - \bar{q}, v^2 - \bar{q}}$  the ring morphism  defined by 
$u \mapsto u+v 
$.  Since this maps $u^3$ to $4\bar{q}(u+v)$, so $\ker (+^\sharp)$ is generated by the element $u(u^2 - 4\bar{q})$ and the coimage of $+$ is isomorphic to $R[u]/(u(u^2 - 4\bar{q}))$.
\end{proof}

\begin{remark}Alternatively the curve $+(S\times_{\Sigma}S)\subset |K|$  can be seen directly from the defining conditions for the curves as follows:  denote the tautological sections for the two copies of $K_{\Sigma}$ by $\eta_1$ and $\eta_2$. The two copies of $S$ are thus defined by $\eta_i^2-q=0$ where $q\in H^0(K^2)$ (for $i=1,2$). The curve $+(S\times_{\Sigma}S)\subset |K|$ is thus defined by the conditions $
\eta=\eta_1+\eta_2$, and  $\eta_1^2-q=0$, and $\eta_2^2-q=0$. It follows that $\eta(\eta^2-4q)$ can take the values 
$(\pm\sqrt{q}\pm\sqrt{q})$, or $(2(\pm\sqrt{q})$, or $(\pm\sqrt{q})-2q)$, 
 corresponding to the four cases $(\eta_1,\eta_2)=(\pm\sqrt{q},\pm\sqrt{q})$. Evaluating these cases we get $\eta(\eta^2-4q)=0$ since
\begin{equation}
\eta(\eta^2-4q)=\begin{cases}
(\pm2\sqrt{q})(0)\ \mathrm{in\ the\ cases}\ (+,+)\ \mathrm{or}\ (-,-)\\
\quad (0)(-4q)\ \mathrm{in\ the\ cases}\ (+,-)\ \mathrm{or}\ (-,+).
\end{cases}
\end{equation}
\end{remark}

\noindent  The fact that scheme-theoretically $+(S\times_{\Sigma}S)=\Sigma\cup S'$ can be understood in light of Proposition \ref{phinotphi2}.  More precisely:

\begin{proposition}\label{Snot2S} Let $(4q,0)\in H^0(K^2)\oplus H^0(K^2)$ be a point in the $\SO(4,\C)$-Hitchin base, and let $X=2\Sigma\cup S'$ be the spectral curve in the fiber over $(4q,0)$. Let $(E,\Phi)$ define a point in $h_{\SO(4,\C)}^{-1}(4q,0)\cap\mathcal{M}^+(\SO_0(1,3)$ whose spectral data is $(X,\calE)$ for  $\calE$ a  rank one torsion free sheaf on $X$. Then the scheme-theoretic support of any $\calE$ is $X_{red}=\Sigma\cup S'$.
\end{proposition}

\begin{proof} 
This follows directly from Proposition \ref{phinotphi2}.
\end{proof}

\noindent Proposition \ref{Snot2S}  means that every spectral sheave $\calE$ is of the form $i_*\calF$ for some coherent sheaf $\calF$ on $X_{red}$ and with $i: X_{red} \hookrightarrow X$  the natural inclusion.  In other words, $+(S\times_{\Sigma}S)$ is the spectral curve for an $\SO_0(1,3)$-Higgs bundle in the sense that it is the support of the rank one torsion free sheaves  from \cite{simpson92}.  \\

\subsubsection{The spectral bundles}

We have not fully succeeded in relating the fiber product construction to the $\SO_0(1,3)$- spectral data described in   the previous sections but we expect that the following observations are steps in that direction.  Let $(S,L)$ be the spectral data defining the $\SL(2,\C)$-Higgs bundle $(V,\varphi)$ or, equivalently, let $((S,L),(S,L^{*}))$ be the spectral data defining the $\SL(2,\C)\times\SL(2,\C)$-Higgs bundle $(V,\varphi), (V,-\varphi))$. Through the projections $p_i$ to each copy of $S$ let \begin{eqnarray}\mathcal{L}:&=&p_1^*(L)\otimes p_2^*(\sigma^*(L))\nonumber\\&=&p_1^*(L)\otimes P_2(L^{*})\nonumber\end{eqnarray} and let $\mathcal{L}_{\pm}:=\mathcal{L}|_{\Delta_{\pm}}$. Notice that with $\sigma_4=\sigma_{12}\sigma_3$ we get

$$\sigma_4=\begin{cases}1\ on\  \Delta_-\simeq S\\\sigma\ on\ 
\Delta_+\simeq S\end{cases}\quad \textrm{while}\quad \sigma_3=\begin{cases}1\ on\  \Delta_+\\\sigma\ on\ 
\Delta_-\end{cases}\ .$$ 

\noindent Thus
\begin{align*}
(S\times_{\Sigma}S)/\sigma_4&\simeq\Delta_+/{\sigma}\cup\Delta_-\simeq\Sigma\cup S\\
(S\times_{\Sigma}S)/\sigma_3&\simeq\ \Delta_+\cup\Delta_-/\sigma\ \simeq \ S\cup \Sigma\ ,
\end{align*}
\noindent with a map between these two quotients coming from the map $(x,y)\mapsto (x,\sigma(y))$ on $S\times_{\Sigma}S$.  The coherent sheaf $\mathcal{L}$ is invariant under the maps induced by the involutions $\sigma_3$ and $\sigma_{12}$ (and hence also $\sigma_4$). Moreover 
\begin{align}
\mathcal{L}|_{\Delta_+}&\simeq\pi^*(\mathcal{O}_{\Sigma})\simeq\mathcal{O}_S\\
\mathcal{L}|_{\Delta_-}&\simeq\ L^2.
\end{align}

It remains to reconcile the information encoded in $\mathcal{L}$ with the spectral data  of \cite{cayley} for the $\SO_0(1,3)$-Higgs bundle  \[\mathcal{I}_2(V,\varphi)=\left(\mathcal{O}\oplus Sym^2V, \left(\begin{matrix}0&\beta\\-\beta^T&0\end{matrix}\right)\right).\]

It would also be interesting to understand in future work the part of the fiber $h_{\SO(4,\C)}^{-1}(4q,0)$ which does not lie in $\mathcal{M}^+(\SO_0(1,3))$ in terms of fibre products since one could possibly recover spectral sheaves on curves of the form $2\Sigma\cup S'$ using the fiber product construction of this chapter.


\subsection{Comments on mirror symmetry}\label{dual1}
As mentioned before, through work of Baraglia and Schaposnik there is a conjectural dual space to the moduli space of real $G$-Higgs bundles. 

\smallbreak
     \noindent {\bf Conjecture}. \cite{slices}. {\it 
     The support of the dual brane to $\mathcal{M}(G)$ is the moduli space $\mathcal{M}(^NG)\subset \mathcal{M}(^L \GC)$ of $^NG$-Higgs bundles where $^NG$ is the group associated to the Lie algebra $\check{\mathfrak{h}}$ in \cite[Table 1]{LPS_Nadler}, refer to as the ``Nadler group''.} 
     \smallbreak
     
     Moreover, for the case of orthogonal Higgs bundles with signature, as considered in this paper, it was further conjectured that the sheaf supported on the dual brane should remain the same once the parity of $q$ is fixed. 
The real form $SO_0(1,3)$ is quasi-split, and from \cite{LPS_Nadler}, the Nadler group associated to $SO(1,3)$ is $SO(3,\C)\subset SO(4,\C)$, and associated to $SO_0(1,3)$  is the locus $ \{ g\otimes g \ | \ g\in SL(2,\C) \} \subset SO(4, \C)=SO(\C^2 \otimes \C^2 , \omega \otimes \omega)$,  giving the identification of the Nadler group $\nad{SO_0(1,3)}\cong PGL(2,\C) \subset SO(4,\C)$. 
From the definitions of complex Higgs bundles, an   $^NSO_0(1,3)$-Higgs bundle is of the form $(E\otimes E, \varphi \otimes \Id + \Id\otimes \varphi ),$ where $E$ is a rank $2$ vector bundle on $\Sigma$ with trivial determinant and thus satisfying $E\otimes E \cong \calO_\Sigma \oplus \Sym^2E$. 

Mirror symmetry over smooth fibres of the Hitchin fibration should correspond to the usual Fourier-Mukai transform of abelian varieties (e.g. see \cite{Kap,dopa}). Fixing a theta-characteristic of the base curve, the locus of $SO_0(1,3)$-Higgs bundles in this fibre corresponds to the inclusion of abelian varieties in the short exact sequence 
\begin{align}
    0 \to \Prym^\vee \to \frac{\Prym \times \Prym}{\Prym [2]} & \to \Prym \to 0. \label{sesabvars}\\
    [(L_1,L_2)]&\mapsto L_1L_2. \nonumber
\end{align}

By choosing an ample line bundle we can identify $\Prym^\vee$ with the quotient of $\Prym $ by the finite group $\Pic^0(\Sigma)[2]$. Thus, the dual abelian variety of the quotient of $\Prym \times \Prym$ by the diagonal subgroup $\Pic^0(\Sigma)[2]$ is itself and dualizing the short exact sequence (\ref{sesabvars})
one can see 
 that the dual brane is supported on $\calM^+(SO(3,\C)) \subset \calM^+(SO(4,\C))$. This locus is precisely the image of the map $$\calM (SL(2,\C))\twoheadrightarrow \calM^+(SO(3,\C)) \subset \calM^+(SO(4,\C)),$$ which is given by pairs of the form $\left(\calO_\Sigma \oplus \Sym^2(E), \Phi = \begin{pmatrix}
0 & 0\\
0 & * 
\end{pmatrix}\right)$.

\section{The real forms $\SU(2)\times\SL(2,\R)$ and $SO^*(4)$}\label{case2}

In this section we consider the Higgs bundles for the isogenous real forms $\SU(2)\times\SL(2,\R)\subset \SL(2,\C)\times\SL(2,\C)$ and $\SO^*(4)\subset\SO(4,\C)$,  and examine the map from $\mathcal{M}(\SU(2)\times\SL(2,\R))$ to $\mathcal{M}(\SO^*(4))$ induced by \eqref{mapa2}. 

\subsection{The Higgs bundles}


From Definition \ref{complex} an $\SU(2)$-Higgs bundle is an $\SL(2,\C)$-Higgs bundle with zero Higgs field, i.e. a pair $(E,0)$ where $E$ is a rank two bundle with trivial determinant, and stability is thus equivalent to the stability of $E$ as a vector bundle.
From Definition \ref{realform}, an  $\SL(2,\R)$-Higgs bundle can be defined as a holomorphic line bundle $L$ (i.e. a vector bundle with structure group $\C^*$) together with a Higgs field $\varphi=(\beta,\gamma)$ where $\beta\in H^0(\Sigma, L^{-2} \otimes K)$ and $\gamma\in H^0(\Sigma, L^2\otimes K)$.
 Viewing $\SL(2,\R)$ as a real form of $\SL(2,\C)$, these Higgs bundles define $\SL(2,\C)$-Higgs bundles $(E,\Phi)$ where $E=L\oplus L^*$ and the Higgs field has the form $\Phi={\tiny \left(\begin{array}{cc}0&\beta\\ \gamma&0\end{array}\right)}$.
In this case, if $\deg(L)=0$ then $(L,\beta,\gamma)$ is always semistable as an $\SL(2,\R)$-Higgs bundle; otherwise the semistability condition reduces to the condition that $\beta\ne 0$ (if $\deg(L)<0$) or  $\gamma\ne 0$ (if $\deg(L)>0$).  It follows from these conditions that $|\deg(L)|\le g-1$ if $(L,\beta,\gamma)$ is semistable.

We refer to \cite{bggpsostar} and \cite{nonabelian} for full discussions of $\SO^*(2n)$-Higgs bundles but summarize here the case of $\SO^*(4)$.  An $\SO^*(4)$-Higgs bundle is defined by a triple $(V,\beta,\gamma)$ where $V$ is a rank $2$ vector bundle , and $\beta: V^*\rightarrow V\otimes K$ and $\gamma: V\rightarrow V^*\otimes K$ are  holomorphic endomorphism which   skew-symmetric as a $K$-valued bilinear form on $V^*$ and $V$, respectively. 
Via the inclusion of $\SO^*(4)$ in $\SO(4,\C)$ the data set $(V,\beta,\gamma)$  defines the $\SO(4,\C)$-Higgs bundle 
\begin{equation}\label{SO*4higgs}
    \left(\left(V\oplus V^*, Q=\begin{bmatrix}0&I\\ I&0\end{bmatrix}\right),\Phi=\begin{bmatrix}0&\beta\\ \gamma&0\end{bmatrix}\right).
\end{equation}
 The (semi)stability condition on $(V,\beta,\gamma)$ is  that $V$ is (semi)stable as a rank 2 bundle, and if $\deg(V)\ne 0$ then at least one of $\beta,\gamma$ must  nonzero\footnotemark\footnotetext{If $deg(V)> 0$ then $\gamma\ne 0$ and if $\deg(V)<0$ then $\beta\ne 0$.}. Hence, a semistable $\SO^*(4)$-Higgs bundles must have $|\deg(V)|\le 2g-2$. 

\subsubsection{The induced map on the Higgs bundles}

In what follows we shall consider the restriction of  the map \eqref{mapa2} to the case where $(V_1,\Phi_1)$ comes from the inclusion $\SU(2)\subset\SL(2,\C)$ and  $(V_2,\Phi_2)$ comes from the inclusion $\SL(2,\R)\subset\SL(2,\C)$. For these real Higgs bundles, one has that 

\begin{itemize}
\item $V_1=U$, a semistable rank 2 bundle with $\det(U)=\mathcal{O}$, and $\Phi_1=0$, 
\item $V_2=N\oplus N^{*}$ and $\Phi_2=\begin{bmatrix}0&\beta\\ \gamma&0 \end{bmatrix}$.
\end{itemize}

\begin{proposition}\label{propso4} 
The map \eqref{mapa2} sends $SU(2)\times SL(2,\R)$-Higgs bundles to $SO^*(4)$-Higgs bundles. 
\end{proposition}
\begin{proof}
 With the above notation, the isogeny from \eqref{I2} is given by  
\begin{equation}
\mathcal{I}_2((V_1,\Phi_1),(V_2,\Phi_2))=(UN\oplus UN^{*}, Q, I\otimes\Phi_2)
\end{equation}
  and the orthogonal structure $Q$ comes from the symplectic structures on $V_1$ and $V_2$ (using the identification $\SL(2,\C)=\Sp(2,\C)$). Denote this symplectic structure on $V_1$ by $\Omega$ and use this to identify $U\simeq U^*$. Then, for $W=UN$ one can thus identify 
$
UN \oplus UN^{*}\simeq W\oplus W^*\, 
$
where $\deg(W)=2\deg(N)$. With respect to this description we get that 
\begin{equation}\label{IPhi2}I\otimes\Phi_2=\begin{bmatrix}0&\Omega^*\otimes \beta\\ \Omega\otimes\gamma&0
\end{bmatrix}
\end{equation}
  where $\Omega^*$ is the symplectic form on $U^*$. Taking $\tilde{\beta}=\Omega^*\otimes\beta$ and $\tilde{\gamma}=\Omega\otimes\gamma$, it follows that the triple $(W,\tilde{\beta},\tilde{\gamma})$ defines an $\SO^*(4)$-Higgs bundle. 

Conversely, suppose that $(V,\beta,\gamma)$ is an $\SO^*(4)$-Higgs bundle which has $\deg(V)$ even.  Then as described in \cite{bggpsostar} one can write  
 $V=U\otimes N$ where $U$ is a rank two bundle with trivial determinant and $N$ is a line bundle with $\deg(N)=\deg(V)/2$. 
Moreover,  $\beta=\Omega^*\otimes\tilde{\beta}$ where $\Omega$ is the symplectic form on $U$ coming from the identification $\SL(2,\C)=\mathrm{Sp}(2,\C)$ and $\Omega^*$ is the corresponding symplectic form on $U^*$, and
finally $\gamma=\Omega\otimes\tilde{\gamma}$, hence giving the data for an $SU(2)\times SL(2,\R)$-Higgs bundles.
\end{proof}
\subsubsection{The induced map on the moduli spaces}

The moduli space $\mathcal{M}(\SU(2)\times\SL(2,\R)$ has 2g-1 connected components labelled by an integer $d=\deg(N)$ where $N$ is the line bundle in the description of the $\SL(2,\R)$-Higgs bundle. We denote these components by $\mathcal{M}_d(\SU(2)\times\SL(2,\R)$, so 
\begin{equation}
    \mathcal{M}(\SU(2)\times\SL(2,\R))=\bigcup_{|d|\le g-1}\mathcal{M}_d(\SU(2)\times\SL(2,\R)).
    \end{equation}

 Similarly, the moduli space $\mathcal{M}(\SO^*(4))$ has $4g-3$ components labelled by the degree of the rank two bundle, i.e.
\begin{equation}
    \mathcal{M}(\SO^*(4))=\bigcup_{|d|\le 2g-2}\mathcal{M}_d(\SO^*(4)).
\end{equation}
  From the discussion in the previous section and Proposition \ref{propso4} one has that:

\begin{proposition}\label{I2so*}\label{propo36} Restricted to $\mathcal{M}(\SU(2)\times\SL(2,\R))$ the map \eqref{isomod} defines a surjection onto the components of $\mathcal{M}(\SO^*(4))$ in which the degree of the rank 2 vector bundle is even.  In notation as above the map is defined by
\begin{align}
\mathcal{I}_2:\mathcal{M}_d(\SU(2)\times\SL(2,\R))&\rightarrow \mathcal{M}_{2d}(\SO^*(4))\nonumber\\
    ([U], [L,\beta,\gamma])&\longmapsto  [U\otimes L,\Omega^*\otimes\beta, \Omega\otimes\gamma].\label{induced-UL}
\end{align}

\end{proposition}

\subsection{Spectral data}\label{SO*spectraldata}

In the Hitchin fibration for $\SL(2,\C)\times\SL(2,\C)$ the $\SU(2)\times\SL(2,\R)$-Higgs bundles  lie in fibers over points of the form $(0,q)\in H^0(K^2)\times H^0(K^2)$. If the $\SL(2,\R)$-Higgs bundle is given by $(N,\beta,\gamma)$ then $q=\beta\gamma$.
An $\SO^*(4)$-Higgs bundle of the form $(U\otimes N,\Omega^*\otimes\beta, \Omega\otimes\gamma)$ lies in the fiber of the $\SO(4,\C)$-fibration over $(2q,q)$  where $q=\beta\gamma$, as can be seen by direct computation of $\det(I\otimes \Phi_2-\eta Id)$ using \eqref{IPhi2}.
We can summarize these maps as follows:
 %
%
\vspace{.1in}
\scalebox{.7}{\xymatrix{
&\mathcal{M}(\SL(2,\C)\times\SL(2,\C))\ar[rr]^{\mathcal{I}_2}\ar'[d][dd]^(.4){h_{\SL(2,\C)\times\SL(2,\C)}}&&\mathcal{M}^+(\SO(4,\C))\ar[dd]^{h_{\SO(4,\C)}}\\
\mathcal{M}_d(\SU(2)\times\SL(2,\R))\ar@{^{(}->}[ru]\ar[rr]^(.6){\mathcal{I}_2}\ar[dd]_{h_{\SL(2,\C)}}&&\mathcal{M}_{2d}(\SO^*(4))\ar@{^{(}->}[ru]\ar[dd]\\
&H^0(K^2)\oplus H^0(K^2) \ar'[r]_(.8){(q_1,q_2)\mapsto(2(q_1+q_2),q_1-q_2)}[rr]&&H^0(K^2)\oplus H^0(K^2)&\\
H^0(K^2) \ar[ru]^{q\mapsto (0,q)}\ar[rr] _{q\mapsto q}&&H^0(K^2)\ar[ru]_{q\mapsto (2q,q)}&
}}

In the most generic case $q_2 = q$ has only simple zeros. Thus, $S_1=2\Sigma$ is the split ribbon on $\Sigma$ with conormal bundle $K^{*}$ and $S_2$, which we will denote simply by $S$, is a smooth curve.


Following the description in Section \ref{firstSLnC} for the case $n=2$, the spectral data for a generic $\SL(2,\C)$-Higgs bundle $(E,\Phi)$ is a pair $(S,L)$ where $S$ is a double cover of $\Sigma$ defined by an equation of the form $\eta^2+a_2=0$ (in the notation of Section \ref{firstSLnC}) and $L$ is a line bundle in $\Prym$. We need two special cases:

\begin{enumerate}
\item $\SU(2)$-Higgs bundles: in this case $\Phi=0$ and the spectral curve is defined by the condition $\eta^2=0$. Thus $S$ is the ribbon $2\Sigma$. The rank two bundle on $\Sigma$ may be viewed as a rank on torsion free sheaf on $2\Sigma$.
\item $\SL(2,\R)$-Higgs bundles: on generic fibers of the Hitchin fibration, the spectral curve $S$ is smooth and the line bundle $L$ is a point of order two in $\Prym$.
\end{enumerate}

The spectral data for $\SO(n,\C)$-Higgs bundles has been described in Section \ref{firstSOnC}. The characteristic polynomials defined by $\SO(4,\C)$-Higgs fields are of the form given in \eqref{sopoly}.  For the real forms $\SO^*(2n)$ the spectral data for the corresponding Higgs bundles is described in \cite{nonabelian}. In the case of $\SO^*(4)$, where the Higgs bundles are of the form given in \eqref{SO*4higgs}, the characteristic equations are thus of the form
\begin{equation}\label{so*eqtn}
   |\Phi-\eta Id |= (\eta^2+q)^2
\end{equation}
 where $q\in H^0(K^2)$ is given by $q=-\frac{1}{2}Tr\beta\gamma$. For generic $q$ this defines a smooth curve $S$ in the total space of $K$. The curve $S$ is a 2-fold cover of $\Sigma$ with involution $\sigma$ that switches the sheets of the covering. The rest of the spectral data (see Proposition 2 in \cite{nonabelian}) consists of a rank two holomorphic bundle $U$ on $S$ which is semistable with $\det(U)\simeq \pi^*K$ and satisfies $\sigma^*U\simeq U$ where the induced action on $\det(U)$ is trivial.

\begin{remark}Recall that the base of the Hitchin fibration on $\mathcal{M}(\SO(2n,\C))$ is $\bigoplus_{i=1}^{n-1}H^0(K^{2i})\oplus H^0(K^n)$, where the last factor encodes the Pfaffian of the Higgs field. Thus the spectral curve defined by \eqref{so*eqtn} lies in the fiber over $(2q,q)\in H^0(K^2)\oplus H^0(K^2)$. The non-reduced curve defined \eqref{so*eqtn} is a ribbon on $S$.
\end{remark}

\begin{corollary} With notation as above, if the spectral curve for the $\SL(2,\R)$-Higgs bundle $(L,\beta,\gamma)$ is $S$, then the spectral curve for the $\SO^*(4)$-Higgs bundle $(U\otimes L,\Omega^*\otimes\beta, \Omega\otimes\gamma)$ is the ribbon $2S$.\label{coro14}
\end{corollary}
\label{subsubso14}
\subsection{Interpretation of the maps induced by the isogeny $\mathcal{I}_2$}
In what follows we shall explore to what extent the fiber product mechanism   accounts for the ${\rm SO}^*(4)$-spectral data described in Section \ref{subsubso14}.

\subsubsection{The map on spectral curves}
From Corollary \ref{coro14}, the effect on spectral curves of the map \eqref{induced-UL} is thus to map
\begin{equation}\label{I2onspec-SO*}
    (2\Sigma, S)\longmapsto 2S\ .
\end{equation}
Recall that for generic fibers of the $\SL(2,\C)\times\SL(2,\C)$-Hitchin fibration, i.e. on fibers over $(q_1,q_2)$ such that the corresponding spectral curves $S_1$ and $S_2$ are smooth with disjoint ramification divisors, the map $\mathcal{I}_2$ induces a map on spectral curves which can be described in terms of a fiber product \cite{iso}.  More precisely (see Section \ref{iso11}), $S_1\times_{\Sigma}S_2$ is the normalization of the spectral curve $+(S_1\times_{\Sigma}S_2)$. We now explore what the fiber product construction yields if $S_1=2\Sigma$ and $S_2=S$.

\begin{proposition}\label{2SSigma-blX} \label{propo40}
The curve $2\Sigma \times_\Sigma S$ is a split ribbon on $S$ with ideal sheaf isomorphic to $\pi^*K^{*}$, where $\pi : S \to \Sigma$. It is, moreover, isomorphic to the blow-up $\Bl_{D_0}(2S)$ of the ribbon $2S$ at $D_0=R_\pi$ (considered as a closed subscheme) and the restriction of the map \eqref{plusmap} to $2\Sigma \times_\Sigma S$, i.e.
\begin{equation}\label{plusmap2}
+ : 2\Sigma \times_\Sigma S \to 2S\ ,
\end{equation}
 is the blow-up map.  
\end{proposition}

\begin{proof}
The $\Sigma$-scheme $Y = 2\Sigma \times_\Sigma S$  has dimension $1$ and 
\[Y_{red} \cong ((2\Sigma)_{red} \times_\Sigma S)_{red} \cong (\Sigma \times_\Sigma S)_{red} \cong S_{red} \cong S.\]
Moreover, 
$0\to \calI \to \calO_Y \to \calO_S \to 0$ 
is locally given by the short exact sequence of rings
\[0 \to (u) \to \frac{R[u,v]}{u^2, v^2 - \bar{q}} \to \frac{R[v]}{v^2 - \bar{q}} \to 0\ .\]
Here $u,v$ are nowhere vanishing sections in $ H^0(U, K^{*})$ where $U\subset\Sigma$ is an open affine set. Thus $| K|_U | = \Spec (R[u])$, where $R = \calO_\Sigma (U)$. By $\bar{q}$ we mean the  local function (on $U$) $\bar{q} = \inn{q}{u^2} \in R$ obtained by the natural pairing between $K^2$ and $K^{-2}$. 
To see that $2\Sigma \times_\Sigma S$ is a split ribbon we use the natural projection onto $S$.  

It remains to prove that \eqref{plusmap2} describes a blow-up map. From the local picture it is clear that \eqref{plusmap2} is a set-theoretical bijection which is an isomorphism away from the zeros of $q$. To understand the map at the zeros of $q$ we first consider the following identifications:
\begin{align*}
\alpha : \hat{\calO}_{2S,p} = \frac{\C  \llbracket\bar{q}, u \rrbracket }{(u^2-\bar{q})^2} &\to \frac{\C  \llbracket u, \epsilon \rrbracket }{\epsilon^2}&\quad \beta : \hat{\calO}_{Y,p^\prime} \cong \frac{\C  \llbracket     \bar{q}, x,y \rrbracket }{x^2, y^2 - \qb} &\to \frac{\C  \llbracket     \tilde{u}, \tilde{\epsilon} \rrbracket }{\tilde{\epsilon}^2}\\
\qb & \mapsto u^2-\epsilon & x & \mapsto \tilde{\epsilon}/2\\
u & \mapsto u & y & \mapsto \tilde{u} -\tilde{\epsilon}/2\\
&& \qb & \mapsto \tilde{u} (\tilde{u} - \tilde{\epsilon} ) 
\end{align*}
These ring morphisms are  well-defined and are isomorphisms. Locally, under these isomorphisms, the induced map of (\ref{plusmap2}) at the level of completed local rings is 
$
\hat{\calO}_{2S,p} \cong \frac{\C  \llbracket     u, \epsilon \rrbracket }{\epsilon^2}  \to \hat{\calO}_{Y,p^\prime} \cong \frac{\C  \llbracket     \tilde{u}, \tilde{\epsilon} \rrbracket }{\tilde{\epsilon}^2} $, sending $u  \mapsto \tilde{u}$ and $\epsilon  \mapsto \tilde{u}\tilde{\epsilon}$. Hence, one has   precisely the algebra extension corresponding to blowing-up\footnote{In particular, blowing-up commutes with passing to the completed local ring.} the point $p\in S$ (with multiplicity one) as a closed subscheme of $2S$ (for more details see \cite{eis} and \cite[Section 2]{ck}).
\end{proof}
\subsubsection{The map on the spectral bundles}

Viewed inside the moduli space $\mathcal{M}(\SL(2,\C)\times\SL(2,\C)$ the $\SU(2)\times\SL(2,\R)$-Higgs bundles lie in fibers of the Hitchin fibration over $(0,q)$. Having studied the spectral curves for these fibers, in order to describe the rest of the spectral data, and to understand the effect on it of the isogeny, we need to understand the $\SL(2,\C)$-fibers over $0$ (the $\SL(2,\C)$- nilpotent cone) and over $q$.

In $\mathcal{M}(\SL(2,\C))$ the spectral curve for the nilpotent cone is the ribbon $2\Sigma$ and hence from \cite{simpson92} an element in $h^{-1}_{SL(2,\C)}(0)$ is a rank one torsion free sheaf $\calE$ on $2\Sigma$.  There are two possibilities: 
\begin{enumerate}
\item $\calE = i_*E$, where $E\to \Sigma$ is a semi-stable rank $2$ vector bundle with trivial determinant (here $i:\Sigma \hookrightarrow 2\Sigma$), \textit{or}
\item $\calE$ is a generalized line bundle on $2\Sigma$ given as an extension 
\begin{equation*}\label{genline}0\to \bar{\calE}K^{*}(D^\prime)\to \calE \to \bar{\calE} \to 0,
\end{equation*}
 where $D^\prime$ is an effective divisor on $\Sigma$ such that $\bar{\calE}^2(D^\prime)\cong K$.
In particular, $\bar{\calE}$ is a line bundle on $\Sigma$ such that
$
1\leq \deg (\bar{\calE}) \leq g-1.
$
\end{enumerate}
 For the generic $q\in H^0(K^2)$ an element in $h^{-1}_{SL(2,\C)}(q)$ is described by a line bundle $L$ (of degree $\deg (L) = 2(g-1)$) on the spectral curve $S$ where $S$ is defined by $q$ and $L$ satisfies 
$L\pip K^{-1/2}\in \Prym$  (or, equivalently, $\sigma^*L\cong L^*\pip K$).
%
%
The next Proposition shows that we have enough information to describe the induced isogeny map on spectral data for the full fibers over these points.

\begin{proposition}\label{propo41}\label{SU2SL2full} Let $q\in H^0(K^2)$ define a smooth spectral curve $S$. Let $(2\Sigma, \calE)$ and $(S,L)$ be spectral data for $\SL(2,\C)$-Higgs bundles $(E,0)$ and $(V,\varphi)$ in the fibers  $h^{-1}_{SL(2,\C)}(0)$ and $h^{-1}_{SL(2,\C)}(q)$ respectively.  Then,
\begin{align}
 \label{isofibres}
\calI_2 : h^{-1}_{SL(2,\C)}(0)\times h^{-1}_{SL(2,\C)}(q)&\longrightarrow h^{-1}_{SO(4,\C)}(2q,q)\nonumber\\
((2\Sigma,\mathcal{E}),(S, L))&\longmapsto (+(2\Sigma\times_{\Sigma}S),+_*(\calE \boxtimes L)).
\end{align}
\end{proposition}

\begin{proof}

We have the following Cartesian diagram:
\[
\begin{tikzcd}
2\Sigma \times_\Sigma S \arrow{r}{p_2}\arrow{rd}{\tilde{\pi}}\arrow{d}[swap]{p_1} & S \arrow{d}{\pir} \\
2\Sigma  \arrow{r}{\varpi} & \Sigma.
\end{tikzcd}
\]   
  Note that $\varpi$ and $\pir$ are flat and so given $(\mathcal{E},L)\in  h^{-1}_{SL(2,\C)}(2\Sigma) \times h^{-1}_{SL(2,\C)}(S)$ we get $
\varpi^*(\pir_{,*}L) \cong p_{1,*}(p_2^*L)$ and $\pip (p_*\calE) \cong p_{2,*}(p_1^*\calE).$ 
Thus, by the identities above and the projection formula,   
\begin{align*}
\tilde{\pi}_* (\calE \boxtimes L) & = \tilde{\pi}_* (p_1^*\calE \otimes p_2^*L)\\
& \cong \varpi_*\circ p_{1,*}(p_1^*\calE \otimes p_2^*L)\\
& \cong \varpi_* (\calE \otimes p_{1,*}(p_2^*L)) \\
& \cong \varpi_* (\calE \otimes \varpi^*(\pir_{,*}L)) \\
& \cong  \varpi_* \calE \otimes \pir_{,*}L.
\end{align*}

To complete the proof we need to show that 
$\tilde{\pi}_* (\calE \boxtimes L) \simeq \pi_*+_*(\calE\boxtimes L)$.
 It then follows a posteriori from the properties of the induced map on vector bundles that $+_*(\calE\boxtimes L)$ defines spectral data in the sense of Simpson for an $\SO(4,\C)$- Higgs bundle, i.e. for a point in the fiber $h^{-1}_{\SO(4,\C)}$.  Moreover, since  $\tilde{\pi}:= \varpi \circ p_1 = \pir \circ p_2$, locally $\tilde{\pi} = \pi \circ +$, where $\pi:X\to \Sigma$ is the usual finite morphism of degree $4$. Thus $\tilde{\pi}_* (\calE \boxtimes L)\cong \pi_* +_* (\calE \boxtimes L)$, as required. Finally, it should be noted that having shown that the spectral curve in the image of $\mathcal{I}_2$ corresponds to that of an $SO(4,\C)$-Higgs bundle, through the usual direct image procedure (e.g.  \cite{nonabelian}), one obtains an orthogonal Higgs field compatible with the orthogonal structure of the bundles.  
 \end{proof}

The $\SU(2)\times\SL(2,\R)$-Higgs bundles correspond to  setting $\calE=i_*E$ and $U=L\pip K^{-1/2}$ a point of order two in $Prym(S,\Sigma)$.   With  $N_0(SO^*(4))$, $N^+$ and $N(SO(4,\C))$ as in Theorem \ref{fibreso} and Section \ref{so4sect} one has the following. 

\begin{proposition}\label{propo42}
 With the above notation, if $\calE=i_*E$ and $U=L\pip K^{-1/2}$  in  $Prym(S,\Sigma)$, one can write
 
  \begin{equation}\label{E'}
  E\otimes {\pi_{red}}_*L\simeq {\pi_{red}}_*E'\ ,\end{equation}
  
 where   $E'$ is a bundle on $S$ given by  $E'={\pi_{red}}^*E\otimes L$. Moreover $E'$ satisfies the necessary conditions for $(S,E')$ to define a point in $N(SO(4,\C)$.
  \label{propp}
\end{proposition}

\begin{proof}
 The first statement is obvious since $S$ is assumed smooth.  Note that the Prym variety $\Prym$ can be characterized by those line bundles $U\in \Pic^0(S)$ which admit an isomorphism $\alpha : \sigma^*U\cong U^*$.  Since $\pir$ is a ramified double covering, the isomorphism $\alpha$ must satisfy 
$\tp{(\sigma^*\alpha)}=\alpha.$ 
In particular, we have such an isomorphism for $U=L\pip K^{-1/2}$, which induces an isomorphism 
$\alpha : \sigma^*L \to L^*\pip K$ with the same property.  There is thus an  isomorphism $\sigma^*E^\prime \xrightarrow{\alpha} \pip E \otimes L^*$. Moreover, since $E$ is a rank $2$ bundle with trivial determinant, we identify $E$ and its dual $E^*$ via $\omega_E$, so that 
\[\alpha^\prime :\sigma^*E^\prime \xrightarrow{\cong} (E^\prime)^*\otimes\pip K,\]
where $\alpha^\prime = \pip \omega \otimes \alpha$ and
\[\tp{(\sigma^*\alpha^\prime)} = \tp{(\pip \omega \otimes \sigma^*\alpha)} = - \pip \omega \otimes \alpha = -\alpha^\prime.\]  This shows that $E' \in N(SO(4,\C))$.

Finally, recall from Section \ref{sl4sect} that we need to consider Higgs bundles contained in $N^+:=N^+(SO(4,\C))$, where Higgs bundles have $w_2=0$. Hence,  to see that the image  lies inside $N^+$
 consider any line subbundle $M$ of $E$, so that we have 
$0\to M \to E \to M^* \to 0.$ 
The line bundle $L\otimes \pip M$ is a $\sigma$-isotropic subbundle of $E^\prime$ of even degree and so $E'$ is indeed in $N^+$. 
\end{proof}

\begin{remark}  Let $(\pi_{red*}L,\pi_{red*}\eta)$ denote the $\SL(2,\R)$-Higgs bundle defined by the spectral data $(S,L)$.  Proposition \ref{propo42}  shows that the morphism $\calI_2$ maps 

\begin{align}
((E, 0),(\pi_{red*}L,\pi_{red*}\eta))&\mapsto (E\otimes \pir_{,*}L , \omega_E\otimes \omega_L,  \Id \otimes \pir_{*}\eta)\nonumber \\
&= \pir_{,*} 
(E^\prime,\eta)
\end{align}

\noindent where  $\omega_E$ and $\omega_L$ are the symplectic forms on $E$ and $\pi_{red,*}L$ respectively, coming from the identification $\SL(2,\C)\simeq\Sp(2,\C)$.  Thus, using nonabelian spectral data on smooth curves for $\SU(2)$-Higgs bundles and $\SO^*(4)$-Higgs bundles (as described in Section \ref{SO*spectraldata}), the map on spectral data for generic fibers of the Hitchin fibrations  may be described by

\begin{equation}
((\Sigma, E),(S,L))\mapsto (S,E').\nonumber 
\end{equation}

Alternatively, if we regard the spectral curves to be the non-integral curves defined by the characteristic equations, then the spectral data for the $\SU(2)$-Higgs bundle becomes $(2\Sigma,\calE=i_*E)$. Combining Proposition \ref{SU2SL2full} and Proposition \ref{propp} we   expect that the map induced by $\calI_2$ can be described as 

\begin{equation}
((2\Sigma, \calE),(S,L))\mapsto (2S,(+_*(\calE\boxtimes L)).\nonumber 
\end{equation}

\end{remark}

We shall conclude by considering the case in which $\calE$ is a generalized line bundle as in \eqref{genline}. Recall from Section  \ref{so4sect} 
that the corresponding irreducible components of $h_{SO(4,\C)}^{-1}(2q,q)$ are the Zariski closures of $A_d(SO(4,\C))$, for $1 \leq d \leq 2(g-1)$, and $N(SO(4,\C))$.  By Proposition \ref{SOPrym}  these components can be described as fibrations over symmetric products of $\Sigma$, in which the fibers are torsors for a Prym schemes $\prym (Y_D, \bar{Y}_D)$.  
 Exploiting the specific structure of the ribbon $2\Sigma\cup S$ reveals the abelian structure of these Prym schemes:

\begin{proposition} \label{propo44}With notation as in Section \ref{so4sect}  and Theorem \ref{fibreso},  the Prym scheme $\prym (Y_D , \bar{Y}_D)$ is a smooth and connected commutative algebraic group of dimension 
\[ \dim (\prym (Y_D , \bar{Y}_D)) = \dim SO(4,\C)(g-1) \]
which fits intothesmooth fibration  
\begin{equation}
0\to H^1(\Sigma, K^{*}(D^\prime)) \to \prym (Y_D , \bar{Y}_D) \to \Prym \to 0.
\label{fibrat}
\end{equation}
\end{proposition}
\begin{proof}
From the exponential sequence (see e.g. \cite[Section 9.2]{bosch}) we obtain 
\[0\to H^1(S, \pip K^{-1}(D^\prime)) \to \Pic^0 (Y_D) \to \Pic^0(S) \to 0.\]
Since the involution $\tilde{\sigma}_D$ lifts $\sigma$, the restriction of a line bundle on $Y_D$ to $S$ above induces the epimorphism 
\begin{equation}
\prym (Y_D , \bar{Y}_D) \to \Prym \to 0.
\label{epi}
\end{equation}
Now, there is an involution on the space of global sections of $\pip K^{-1}(D^\prime)$ induced by the canonical linearization on $\pip K^{-1}(D^\prime)$ and the Prym condition tell us that the kernel of (\ref{epi}) is precisely the $-1$-eigenspace of $H^1(S, \pip K^{-1})$.  The proof can be completed by observing that  
\begin{align*}
\calO_\Sigma \oplus K^{-1} & \xrightarrow{\cong} \pir_{,*}\calO_S\\
(s_1,s_2)& \mapsto \pip s_1 + \lambda \pip s_2
\end{align*}
and $\sigma^*\lambda = -\lambda$. Finally, the dimension count  follows directly from (\ref{fibrat}) since  
$
\dim (\prym (Y_D , \bar{Y}_D))  = g(S) - g +h^1(\pip K^{-1}(D^\prime))$ which is  given by  $
g(\bar{Y}) - g(\bar{Y}_D) 
 = \dim SO(4,\C)(g-1)$, as expected. 
\end{proof}


\subsection{Comments on mirror symmetry}\label{dual2}

We have studied  the isogeny 
 $\mathcal{I}_2$
 which   maps   $SU(2)\times SL(2,\R)$ onto the real form $SO^*(4)$ of $SO(4,\C)$. Considering  the double cover of $Spin(4,\C)$ one can study the conjectural support of the dual brane associated to the Higgs bundle moduli space for $SO^*(4)$. In what follows we shall make some remarks concerning the conjecture of \cite{slices} on the support of the dual brane, with the hope to shed light on some of its aspects. 
 
In this case the Nadler group $\nad{SO^*(4)}$ is isomorphic to $SL(2,\C)$ and sits inside $SO(4,\C)=SO(\C^2\otimes \C^2, \omega \otimes \omega)$ as 
\[\{1\otimes g \ | \ g\in SL(2,\C) \}\subset SO(4,\C).\]
Thus, conjecturally from \cite{slices}, the dual brane should be supported on points in the dual fibre whose underlying bundle is of the form $\calO_\Sigma^{\oplus 2}\otimes E$, where $E$ is a rank $2$ bundle with trivial determinant, and Higgs field have the form $\Id \otimes \varphi$, for $\varphi \in H^0(\Sigma , \End_0E \otimes K)$. In other words, the locus of the dual brane should correspond to Higgs bundles of the form $(E,\varphi)\oplus (E,\varphi)$. In a generic fibre, we can see this locus as 
\begin{eqnarray}\Prym &\cong& \{ [(L,\calO_\Sigma^{\oplus 2})] \ | \ L \in \Prym  \} \subset N\nonumber\\&\cong& \dfrac{\Prym \times \ \calU_\Sigma (2,\calO_\Sigma)}{\Pic^0(\Sigma)[2]}.\nonumber\end{eqnarray}
\begin{remark}\label{difficult}
This case presents two extra difficulties. One comes from the fact that $SO^*(4)$-Higgs bundles never intersect smooth fibres of the $SO(4,\C)$-Hitchin fibration (in particular there is no known Fourier-Mukai transform for sheaves on non-reduced curves). The other difficulty is that the locus $N_0$ of $SO^*(4)$-Higgs bundles inside this fibre is not connected. \end{remark}

It is the hyperholomorphic sheaf which is expected to take care of the second difficulty  in Remark \ref{difficult}, and the dual brane to each connected component of the locus of $SO^*(4)$-Higgs bundles should be supported on the same hyperk\"{a}hler subvariety. Hence, in terms of support, the main difficulty to find the dual brane comes from the fact that there is no known version of Fourier-Mukai providing an autoduality between these compactified Jacobians. However, one could construct the dual support thorough the annihilator space inside the dual fibre. 

\subsection*{Acknowledgments} The authors are thankful for financial support from from U.S. National Science Foundation grants DMS 1107452, 1107263, 1107367 ÒRNMS: GEometric structures And Representation varietiesÓ (the GEAR Network) which financed several research visits during which the paper was written. The authors are thankful to J. Heinloth, N.~Hitchin, S.~Rayan, R.~Rubio, and A.~Thompson for useful conversations. L.P. Schaposnik is partially supported by the NSF grant DMS-1509693, the NSF CAREER Award DMS-1749013, and by the Alexander von Humboldt Foundation. This material is also based upon work supported by the National Science Foundation under Grant No. DMS- 1440140 while SB and LPS were in residence at the Mathematical Sciences Research Institute in Berkeley, California, during the Fall 2019 semester.

\bibliography{Schaposnik_Laura_July2019}{}

\begin{thebibliography}{10}

\bibitem{N2}
Nigel Hitchin.
\newblock Stable bundles and integrable systems.
\newblock {\em Duke Math. J.}, 54(1):91--114, 1987.



\bibitem{syz}
Eric Zaslow, Andrew Strominger, and S-T.~Yau.
\newblock { Mirror symmetry is T-duality}.
\newblock {\em Nuclear Physics B}, 479(1-2):243--259, 1996.

\bibitem{uhlenbeck1982}
Karen Uhlenbeck.
\newblock Connections with $L^p$ bounds on curvature
\newblock {\em Comm. Math. Phys.}, 83(1): 31--42, 1982.

\bibitem{iso}
Steven Bradlow and Laura~P. Schaposnik.
\newblock Higgs bundles and exceptional isogenies.
\newblock {\em Res. Math. Sci.}, 3:Paper No. 14, 28, 2016.

\bibitem{emilio3}
Emilio Franco and Marcos Jardim.
\newblock Mirror symmetry for Nahm branes.
\newblock arXiv:1709.01314






\bibitem{thesis}
Laura~P. Schaposnik.
\newblock Spectral data for {G}-{H}iggs bundles.
\newblock {\em DPhil Thesis --University of Oxford (United Kingdom)}, page~1,
  2013.
 


\bibitem{donagi1993decomposition}
Ron Donagi.
\newblock Decomposition of spectral covers.
\newblock {\em Ast{\'e}risque}, 218:145--175, 1993.

\bibitem{doga}
Ron~Donagi and Dennis~Gaitsgory.
\newblock The gerbe of higgs bundles.
\newblock {\em Transform. Groups}, 7(2):109--153, 2001.

\bibitem{GPPNR}
Oscar Garc\'{\i}a-Prada, Ana Peon-Nieto, and S.~Ramanan.
\newblock Higgs bundles for real groups and the hitchin-kostant-rallis section.
\newblock {\em Trans. of the Amer. Math}, 370(4):2907--2953, 2018.

\bibitem{GPPN}
Oscar Garc\'{\i}a-Prada and Ana Peon-Nieto.
\newblock Higgs bundles, abelian gerbes and cameral data.
\newblock arXiv:1902.06139

\bibitem{simpson92}
Carlos~T. Simpson.
\newblock Higgs bundles and local systems.
\newblock {\em Inst. Hautes \'{E}tudes Sci. Publ. Math.}, 75:5--95, 1992.

\bibitem{cayley}
David Baraglia and Laura~P. Schaposnik.
\newblock Cayley and {L}anglands type correspondences for orthogonal {H}iggs
  bundles.
\newblock {\em Trans. Amer. Math. Soc.}, 371(10):7451--7492, 2019.

\bibitem{nonabelian}
Nigel Hitchin and Laura~P. Schaposnik.
\newblock Nonabelianization of {H}iggs bundles.
\newblock {\em J. Differential Geom.}, 97(1):79--89, 2014.

\bibitem{Kap}
Anton Kapustin and Edward Witten.
\newblock Electric-magnetic duality and the geometric {L}anglands program.
\newblock {\em Commun. Number Theory Phys.}, 1(1):1--236, 2007.

\bibitem{slices}
David Baraglia and Laura~P. Schaposnik.
\newblock Real structures on moduli spaces of {H}iggs bundles.
\newblock {\em Adv. Theor. Math. Phys.}, 20(3):525--551, 2016.

\bibitem{N1}
Nigel Hitchin.
\newblock The self-duality equations on a riemann surface.
\newblock {\em Proc. London Math. Soc. (3)}, 55(1):59--126, 1987.

\bibitem{BNR}
Arnaud Beauville, M.~S. Narasimhan, and S.~Ramanan.
\newblock Spectral curves and the generalised theta divisor.
\newblock {\em J. Reine Angew. Math.}, 398:169--179, 1989.

\bibitem{mas1}
Laura~P Schaposnik.
\newblock A geometric approach to orthogonal higgs bundles.
\newblock {\em European Journal of Mathematics}, 4(4):1390--1411, 2018.

\bibitem{mas2}
Nigel Hitchin, Laura~P Schaposnik, et~al.
\newblock Nonabelianization of higgs bundles.
\newblock {\em Journal of Differential Geometry}, 97(1):79--89, 2014.

\bibitem{mas3}
Laura~P. Schaposnik.
\newblock {Spectral Data for U(m,m)-Higgs Bundles}.
\newblock {\em International Mathematics Research Notices},
  2015(11):3486--3498, 03 2014.

\bibitem{mas4}
David Baraglia and Laura Schaposnik.
\newblock Cayley and langlands type correspondences for orthogonal higgs
  bundles.
\newblock {\em Transactions of the American Mathematical Society},
  371(10):7451--7492, 2019.

\bibitem{mas5}
Ana Pe{\'o}n-Nieto.
\newblock Cameral data for $SU(p+1,p)$-higgs bundles.
\newblock {\em arXiv preprint arXiv:1506.01318}, 2015.

\bibitem{simpson}
Carlos~T. Simpson.
\newblock {Moduli of representations of the fundamental group of a smooth
  projective variety. II.}
\newblock {\em Inst. Hautes E\'tudes Sci. Publ. Math.}, 80:5--79, 1995.

\bibitem{N5}
Nigel Hitchin.
\newblock Lie groups and teichm\"{u}ller space.
\newblock {\em Topology}, 31(3):449--473, 1992.

\bibitem{lucas}
Lucas Branco.
\newblock Higgs bundles, Lagrangians and mirror symmetry.
\newblock {\em DPhil. Thesis, University of Oxford}, 2017.

\bibitem{N3}
Nigel Hitchin.
\newblock Langlands duality and {$G_2$} spectral curves.
\newblock {\em Q. J. Math.}, 58(3):319--344, 2007.

\bibitem{simpson88}
Carlos~T. Simpson.
\newblock Constructing variations of {H}odge structure using {Y}ang-{M}ills
  theory and applications to uniformization.
\newblock {\em J. Amer. Math. Soc.}, 1(4):867--918, 1988.

\bibitem{LPS_Nadler}
David Nadler.
\newblock Perverse sheaves on real loop {G}rassmannians.
\newblock {\em Invent. Math.}, 159(1):1--73, 2005.

\bibitem{dopa}
Ron Donagi and Tony Pantev.
\newblock Langlands duality for {H}itchin systems.
\newblock {\em Invent. Math.}, 189(3):653--735, 2012.

\bibitem{bggpsostar}
Steven Bradlow, Oscar Garc\'{\i}a-Prada, and Peter Gothen.
\newblock Higgs bundles for the non-compact dual of the special orthogonal
  group.
\newblock {\em Geom. Dedicata}, 175:1--48, 2015.

\bibitem{eis}
David Eisenbud and Mark Green.
\newblock {Clifford indices of ribbons}.
\newblock {\em Trans. Amer. Math. Soc.}, 347(3):757--765, 1995.

\bibitem{ck}
Dawei Chen and Jesse Kass.
\newblock {Moduli of generalized line bundles on a ribbon}.
\newblock {\em J. Pure Appl. Algebra}, 220(2):822--844, 2016.

\bibitem{bosch}
Siegfried Bosch, Werner L\"utkebohmert, and Michel Raynaud.
\newblock {N\'eron models}.
\newblock {\em Ergebnisse der Mathematik und ihrer Grenzgebiete}, 21(3), 1990.

\end{thebibliography}
\bibliographystyle{unsrt}
  
\address{{\bf  Steve Bradlow - {\sc  University of Illinois at Urbana-Champaign, USA.}\\
\tt  bradlow@illinois.edu}}

\address{ {\bf  Lucas Branco - 
\tt lucasmpcastello@gmail.com  }}

\address{{\bf  Laura P. Schaposnik - 
{\sc  University of Illinois at Chicago, USA.}  \\  \tt  schapos@uic.edu}}

\end{document}